\documentclass[10pt,english,reqno]{amsart}
\usepackage{amsmath,amssymb,amsthm,amsfonts,graphicx,color}
\usepackage{amssymb,geometry}

\usepackage{tikz}

\usepackage[T1]{fontenc}
\usepackage[latin1]{inputenc}
\usepackage[english]{babel}
\usepackage{geometry}

\usepackage{cases}

\usepackage[colorlinks=true]{hyperref}

\usepackage[colorlinks=true]{hyperref}
\hypersetup{linkcolor=black,citecolor=black,filecolor=black,urlcolor=black} 

\newcommand{\red}[1]{{\color{red} #1}}

\newcommand{\x}{{\tt x}}
\newcommand{\y}{{\tt y}}

\newcommand{\N}{\mathbb{N}}

\newcommand{\Q}{\mathbb{Q}}
\newcommand{\R}{\mathbb{R}}

\newcommand{\eps}{\varepsilon}

 \newtheorem{lema}{Lemma}[section]
 \newtheorem{lemma}{Lemma}[section]
\newtheorem{definition}{Definition}[section]
\newtheorem{theorem}{Theorem}[section]
\newtheorem{proposition}{Proposition}[section]
\newtheorem{corollary}{Corollary}[section]
\newtheorem{remark}{Remark}[section]
\newcommand{\bremark}{\begin{remark} \em}
\newcommand{\eremark}{\end{remark} }

\numberwithin{equation}{section}

\title{The Allen-Cahn equation and general minimal cones}

\author{Oscar Agudelo}
\address{Universidad Nacional de Colombia Sede Medell\'{i}n\\ Campus el Volador}
\email {oiagudel@unal.edu.co}

\author{Matteo Rizzi}
\address{Universit\'{a} degli Studi di Bari Aldo Moro, via Orabona 4, 70125, Bari, Italy.}
\email{matteo.rizzi@uniba.it, mrizzi1988@gmail.com}
\thanks{The research of Oscar Agudelo was supported by Universidad Nacional de Colombia,
Sede Medell\'{i}n, Facultad de Ciencias, Departamento de Matem\'{a}ticas, Grupo de Investigaci\'{o}n en Matem\'{a}ticas
de la Universidad Nacional de Colombia Sede Medell\'{i}n, Proyecto "Problemas lineales y no lineales de autovalores y aplicaciones a las ecuaciones en derivadas parciales", Fondo de Investigaci\'{o}n de la Facultad de Ciencias, Proyecto Hermes codigo 67035. The research of Matteo Rizzi was partially supported by the Università degli studi di Bari Aldo Moro, through a TNE project}

\begin{document}

\begin{abstract}
We construct solutions to the Allen-Cahn equation $\Delta u+u-u^3=0$ in $\R^{N+1}$, with $N\ge 3$, whose zero level set is asymptotic (at infinity) to a given minimal cone. Our construction is quite general, since we only use a suitable nondegeneracy assumption; neither stability nor symmetry of the underlying cone is required. Moreover, we calculate the morse index of such solutions. In particular, our main result generalizes a result of \cite{pacard2013stable} dealing with stable solutions to the Allen-Cahn equation in high dimensions. More precisely, we give a positive answer to a remark raised in \cite{pacard2013stable} about the possibility to relax the assumption for the underlying cone is area-minimizing or at least stable. As a corollary, we also obtain solutions with infinite morse index and connected zero level set.
\end{abstract}
\maketitle

\section{Introduction}

This work is concerned with entire bounded solutions of the classical semilinear elliptic equation
\begin{equation}
\label{allen-cahn}
\Delta u-W'(u)=0\quad\hbox{in} \quad \R^{N+1},
\end{equation}
where $N\geq 3$, $W\in C^2(\R)$ is even and satisfies $W(\pm a)=0$, for some $a>0$, $W(s)>0$ for $s\neq \pm a$, and $W''(\pm a)>0$. A well-known and well explored example of such $W$ is given by nonnegative balanced double well potential $$
W(s)=\frac{(1-s^2)^2}{4} \quad \hbox{with} \quad W'(s)=-s+s^3,\quad s\in \R,
$$
which corresponds to the Allen-Cahn equation.

Let $C$ be a given minimal cone $C$ in $\R^{N+1}$. In this work, we establish conditions on $C$ that guarantee existence of solutions to \eqref{allen-cahn} with the property that their nodal set is asymptotic, in an appropriate sense, to the cone $C$.

In principle, we require neither symmetry nor stability assumptions about the given cone $C$. These two conditions are required in several works, in order to study geometric properties of such cones (see  \cite{alencar2005m,hardt1985area,Mazet_2017} and references therein), or in some cases to obtain solutions of semilinear equations of the type \eqref{allen-cahn} (see \cite{dang1992saddle,CABRETERRA2009,CABRETERRA2010}).

\medskip

In this last regard, the authors in \cite{pacard2013stable} start with an area minimizing cone and make use of the standard {\it infinite dimensional Lyapunov-Schimdt reduction method} to find stable solutions (see Definition \ref{def:morse_index_allen-cahn} in page \pageref{def:morse_index_allen-cahn}) having the property at which we are aiming.

As we will see, the minimal requirements on the underlying cone yield both new results and some of the related known results in the literature. To be more precise, the results hereby presented complete and complement the results in \cite{pacard2013stable,agudelo2022doubling,agudelo2022k,agudelo2024jacobi} related to this topic. We remark that the construction in \cite{pacard2013stable} is based on the stability of the underlying cone, while the constructions in \cite{agudelo2022doubling,agudelo2022k,agudelo2025fredholm} are based on symmetry. For these reasons, this work is a relevant generalization.\\

In \cite{del2013entire} and for $N=2$, the authors showed the existence of solutions to \eqref{allen-cahn} having zero level set close to a large dilation of an embedded minimal surface with finite total curvature, satisfying certain nondegeneracy assumptions. This result relies on the nondegeneracy of the underlying surface rather than on symmetry assumptions. Nonetheless, an example of such surfaces is the catenoid {\it} and this case yields axially symmetric solutions of \eqref{allen-cahn}. Using symmetry arguments and again nondegenaracy assumptions in the symmetric geometric setting, the previous result was generalized to the case $N \geq 3$ in \cite{agudelo2016vyvsvsi}, where the authors built an axially symmetric solution to \eqref{allen-cahn} whose nodal set is  close to a large dilation of a logarithmic correction of the higher dimensional catenoid. This logarithmic correction, which is governed by the Liouville equation, is needed due to the fact that outside a large ball, the higher dimensional catenoid is asymptotic to two parallel planes.\\ 

Let us now introduce some notation. For any $n\in \mathbb{N}$, $\mathbb{S}^{n-1}$ denotes the unit sphere in $\R^{n}$ centered at the origin. Also,  $O(n)$ denotes the group of orthogonal transformations of $\R^n$. Let $S$ be a subgroup of $O(n)$. A set $D\subset \R^{n}$ is $S-$invariant if for any $\rho\in S$, $\rho(D)\subset D$. If $D$ is an $S-$invariant set, a function $f:D\to \R$ is $S-$invariant if for any $\rho\in S$, $f\circ\rho=f$ in $D$. We remark that $S$ is allowed to be the trivial subgroup.\\ 

Throughout this work $N\ge 3$ and $C\subset \R^{N+1}$ is a minimal cone of codimension one, i.e.,  
$$
C:=\{r\theta\in \R^{N+1}\,:\, r>0 \quad \hbox{and} \quad \theta \in \Gamma\},
$$
where $\Gamma\subset \mathbb{S}^{N}$ is a minimal submanifold of $\mathbb{S}^N$ of codimension one. We say that $C$ is {\it trivial} if $C$ is a hyperplane, i.e., a subspace of $\R^{N}$ of dimension $N$.\\

Let $\nu_{\Gamma}$ denote a fixed choice of the unit normal vector of $\Gamma$ tangent to $\mathbb{S}^{N}$. From the homogeneity of the cone, for $r\theta \in C$, $\nu_C(r\theta):=\nu_{\Gamma}(\theta)$ is a fixed choice of the unit normal vector to $C$ in $\R^{N+1}$ at the point $r\theta$. In what follows we write for short $\nu_{C}(r\theta)=\nu_C(\theta)$. 

Let $\Sigma\subset \R^{N+1}$ be a smooth, complete and embedded minimal hypersurface of $\R^{N+1}$ of codimension one let $S\subset O(N+1)$ be a subgroup. The hypersurface $\Sigma$ is the base object for the nodal set of the solutions we are looking for.

For $R>0$, we write
$$
\Sigma_R:= \Sigma \cap \big(\R^{N+1}\setminus \overline{B_{R}(0)}\big) \qquad \hbox{and}\qquad K_R:= \Sigma \setminus \Sigma_R.
$$

We assume the following hypotheses on $C$ and $\Sigma$: 
\begin{itemize} 
\item[(H1)] $C$ is a nontrivial and $S-$invariant cone;
\vskip 3pt
\item[(H2)] $\Sigma$ is $S-$invariant;
\vskip 3pt
\item[(H3)] $\Sigma$ is asymptotic to the cone $C$, that is, there exist $R>0$ and a smooth function $w:(R,\infty)\times\Gamma\to \R$ such that $\lim \limits_{r\to \infty} \|w(r,\cdot)\|_{C^2(\Gamma)}=0$ and such that
\begin{equation}\notag
\Sigma_R=\{r\theta+w(r,\theta)\nu_C(\theta):\,\theta\in\Gamma,\,r>R\}.
\end{equation}
and $K_R:=\Sigma\setminus\Sigma_R$ is compact.
\color{black}\item[(H4)] $S$ is maximal with respect to property (H1). In other words, if $\rho\in O(N+1)$ is such that $\rho(C)\subset C$, then $\rho\in S$.
\end{itemize}
Note that $S$ could not be the trivial group in the case there are symmetries.\\

Since $\Sigma$ is complete, embedded and smooth, from hypothesis (H3), $\R^{N+1}\setminus \Sigma$ is has two connected components, which we denote by $\Omega_{\pm}$. We will see after the statement of our main result that there is no loss of generality in assuming that $C$ is nontrivial.\\

In order to use such hypersurfaces to construct solutions to the Allen-Cahn equation, we need to introduce some nondegeneracy assumption. For this purpose, we introduce the Jacobi operator $J_\Sigma:=\Delta_\Sigma+|A_\Sigma|^2$, where $\Delta_\Sigma$ is the Laplace-Beltrami operator of $\Sigma$ and $|A_\Sigma|^2$ is the squared norm of the second fundamental form $A_\Sigma$ of $\Sigma$. 

Following \cite{agudelo2025fredholm}, for $\nu\in\R$, we define the spaces
\begin{equation}\notag
D_\nu(\Sigma):=\{\phi\in C^2(\Sigma):J_\Sigma \phi=0,\,\phi(1+|y|)^{-\nu}\in L^\infty(\Sigma)\}
\end{equation}
and
\begin{equation}\notag
D_\nu(\Sigma,S):=\{\phi\in D_\nu(\Sigma):\,\phi\text{ is $S$-invariant}\}.
\end{equation}
The distributional solutions $\phi\in \mathcal{D}'(\Sigma)$ of $J_{\Sigma}\phi=0$ are known as the \textit{Jacobi fields} of $\Sigma$. In particular, we say that a Jacobi field is \textit{geometric} if it is generated by either of the geometric invariances of the zero mean curvature equation, that is dilations, rotations or translations. We denote the space of the geometric Jacobi fields by $G(\Sigma)$ (see also \cite{agudelo2024jacobi}). Moreover, we set
\begin{equation}\notag
\bar{\nu}:=\inf\{\nu\in\R:\, \limsup_{|y|\to\infty}|y|^{-\nu}|\zeta_0(y)|=0\},
\end{equation}
where $\zeta_0(y):=y\cdotp\nu_\Sigma(y)$ is the Jacobi field arising from dilations. From Lemma 3.1 in \cite{agudelo2025fredholm} $\bar{\nu}$ is well defined and finite, and $\bar{\nu}\leq 0$.

\begin{definition}\label{def:S_dila_nond} 
\begin{itemize}
    \item[i.] The minimal surface $\Sigma$ is dilation nondegenerate if $\bar{\nu}>2-N$ and $D_\nu(\Sigma)=\{0\}$ for any $\nu<\bar{\nu}$.

    \vskip 3pt
    \item[ii.] The minimal surface $\Sigma$ is $S$-dilation nondegenerate if $\bar{\nu}>2-N$ and $D_\nu(\Sigma,S)=\{0\}$ for any $\nu<\bar{\nu}$.\end{itemize}
\end{definition}
Property (ii) of Definition \ref{def:S_dila_nond} is the appropriate notion of nondegeneracy which enables us to construct solutions to equation (\ref{allen-cahn}) vanishing close to a dilated copy of $\Sigma$. It is called $S$-dilation-nondegeneracy since in many known examples, the Jacobi field $\zeta_0(y):=y\cdotp \nu_\Sigma(y)$ coming from dilations fulfills $\zeta_0(y)=O(|y|^{\bar{\nu}})$ as $|y|\to\infty$. This is the case if, for instance, the asymptotic cone is stable (see Remark 
$3.2$ in \cite{agudelo2022k}). Heuristically, the $S$-dilation-nondegeneracy is equivalent to saying that there are no $S$-invariant Jacobi fields which decay faster than $\zeta_0$. In section 3 we present some examples of such cones $C$ and related groups of symmetry $S\subset O(N+1)$ and minimal hypersurfaces $\Sigma$ that satisfy our set of assumptions.\\

Next, let $u\in C^2(\R^{N+1})$ be a solution of \eqref{allen-cahn} and consider the linearized operator of \eqref{allen-cahn} around $u$, namely
$$
L_u \phi:=-\Delta \phi + W''(u) \phi\quad \hbox{in } \R^{N+1},
$$
and its associated quadratic form 
\begin{equation}
\label{def:quadr_sln_u}
\mathbb{Q}_u(\phi):=\int_{\R^{N+1}}\Big(|\nabla \phi |^2 +W''(u)\phi^2 \Big)dx \quad \hbox{for } \phi \in C^{\infty}_c(\R^{N+1}).    
\end{equation}

\begin{definition}
\label{def:morse_index_allen-cahn}
The morse index of the solution $u$ is defined as 
\begin{small}
\begin{equation}\label{def:formula_morse_index_allen-cahn}
{\rm Morse}(u):=\sup \Big\{{\rm dim}(Y) \,:\, Y\hbox{ nontrivial vector subspace of }C^{\infty}_c(\R^{N+1}),\,\,  \forall\, \phi \in Y\setminus \{0\}: \Q_u(\phi) <0 \Big\}.
\end{equation}
\end{small}

In the case the set over which the supremum is taken is the empty set, we define ${\rm Morse}(u):= 0$ and we say that $u$ is a stable solution.
\end{definition}

Observe that ${\rm Morse}(u)\in \mathbb{N}\cup\{0,\infty\}$ and roughly speaking, it counts the maximum number of linearly independent directions along which the solution $u$ of \eqref{allen-cahn} looses its locally energy minimizing character.\\

In a similar manner, for the hypersurface $\Sigma$ we consider the linearization of the mean curvature operator around $\Sigma$, namely the Jacobi operator of $\Sigma$
$$
J_{\Sigma}\varphi:= \Delta_{\Sigma} \varphi + |A_{\Sigma}|^2\varphi \quad \hbox{in }\Sigma,
$$ 
where $\Delta_\Sigma$ is the Laplace-Beltrami operator of $\Sigma$ and $|A_{\Sigma}|$ is the norm of the second fundamental form of $\Sigma$. The quadratic form associated to $J_{\Sigma}$ is
\begin{equation}\label{def:quadr_Sigma}
\mathcal{Q}_{\Sigma}(\varphi):=\int_{\Sigma}\Big(|\nabla_{\Sigma} \varphi |^2 -|A_{\Sigma}|^2\varphi^2 \Big)d\sigma \quad \hbox{for } \varphi \in C^{\infty}_c(\Sigma ).
\end{equation}

\begin{definition}
\label{def:morse_index_Sigma}
The morse index of the hypersurface $\Sigma$ is defined as 
\begin{small}
\begin{equation}\label{def:formula_morse_index_Sigma}
{\rm Morse}(\Sigma):=\sup \Big\{{\rm dim}(X) \,:\, X\hbox{ nontrivial vector subspace of }C^{\infty}_c(\Sigma),\,\,  \forall \,\varphi \in X\setminus \{0\}: \mathcal{Q}_\Sigma(\varphi) <0 \Big\}.
\end{equation}
\end{small}

In the case the set over which the supremum is taken is the empty set, we define ${\rm Morse}(\Sigma):= 0$ and we say that $\Sigma$ is a stable minimal hypersurface.
\end{definition}


A well known example,  where the discussion has motivated a great deal of progress in the analysis od PDEs and stability matters is the Allen-Cahn equation
\begin{equation}
\label{allen-cahn-eq}
-\Delta u=u-u^3\qquad\text{in $\R^{N+1}$,}
\end{equation}
which was introduced in \cite{allen1979microscopic} as the prototype equation for the continuous modeling of {\it phase transition phenomena}. In the one dimensional case, equation \eqref{allen-cahn-eq} reads as the ordinary diffe\-rential equation (ODE)
\begin{equation}\label{Allen-Cahn-eqn 1D}
v''+ v-v^3=0 \quad \hbox{in} \quad \R.
\end{equation}

Imposing the conditions $v(\pm \infty)=\pm 1$, equation \eqref{Allen-Cahn-eqn 1D} has a unique (up to translations) solution given by
\begin{equation}\label{Allen-Cahn-eqn 1D Sln}
v_{\star}(t) = \tanh(t/\sqrt{2}) \quad \hbox{for} \quad t\in \R.
\end{equation} Besides, $v_{\star}$ is strictly monotone increasing, since $v_{\star}'>0$ in $\R$, and $\{v_{\star}=0\}=\{0\}$.

\vskip 5pt
Given a unit vector $\vec{a}\in \R^{N+1}$ and $\xi_0 \in \R^{N+1}$, the function
\begin{equation}\label{DeGiorgiconjectsln}
u(\xi):= v_{\star}(t), \quad t= \vec{a}\cdot (\xi -\xi_0)\quad \hbox{for} \quad \xi\in \R^{N+1}
\end{equation}
is an entire, bounded and sign-changing solution of \eqref{allen-cahn-eq}. $u(\xi)$ is  monotone in the direction of  $\vec{a}$ and its nodal set is the hyperplane of equation $\vec{a}\cdot (\xi-\xi_0)=0$. This hyperplane is a trivial example of a cone in $\R^{N+1}$ and in this regard $u(\xi)$ is also trivial, since, up to a translation and a rotation, it  depends on one variable only. In particular, since these solutions have a positive directional derivative, they are stable.

\vskip 5pt
In 1978, E. De Giorgi conjectured  that if $1 \leq N \leq 7$, then any solution to \eqref{allen-cahn-eq} which is monotone in one direction and fulfills $|u|\le 1$ actually satisfies \eqref{DeGiorgiconjectsln}, for some unit vector $\vec{a}$ and some point $\xi_0$. This conjecture was motivated by Modica's intuition and the developments in $\Gamma-$convergence theory (see \cite{de1979convergence}), which showed evidence of the strong connection between the study of bounded solutions \eqref{allen-cahn-eq} and the theory of minimal hypersurfaces.

\vskip 5pt
In \cite{ambrosio2000entire,ghoussoub1998conjecture}, De Giorgi's conjecture was established in dimensions $N+1=2,3$. In \cite{savin2009regularity}, it was proven true in dimensions $4 \leq N+1 \leq 8$, under the additional assumption
\begin{equation}
\label{limpm1}
\lim_{\xi_{N+1}\to \pm \infty} u(\xi',\xi_{N+1}) = \pm 1.
\end{equation}

In \cite{liu2026stable} the authors show that every bounded stable entire solution of \eqref{allen-cahn} in $\mathbb{R}^4$ with the standard double well potential is either constant $\pm 1$ or it is the one-dimensional heteroclinic solutions connecting $\pm 1$. This result yields the full De Giorgi conjecture in $\mathbb{R}^4$. The proof reduces stability to a weak stability inequality on the zero set, then combines the Gauss-Bonnet formula with logarithmic cutoffs, without assuming any a priori energy growth bound.

In \cite{del2011giorgi} and for $N+1\geq 9$, the authors constructed a solution which proves that the De Giorgi conjecture is sharp about the dimension. The construction was based on the nontrivial minimal graph $\Gamma$ built in \cite{bombieri1969minimal}. The solution found in \cite{del2011giorgi} is bounded, monotone in one direction and its zero level set is close to the a large dilation of the surface $\Gamma$.\\

The De Giorgi conjecture under the further assumption that the limits in \eqref{limpm1} are uniform in $x'$ is known as the Gibbons conjecture, which was proved to be true in any dimension in \cite{farina2003rigidity}.\\ 

Our main result in this work is the following.

\begin{theorem}\label{main-th}
Let $N\ge 3$. Let $W\in C^2(\R)$ be an even function such that $W(\pm a)=0<W(t)$ for any $t\in\R\setminus\{\pm a\}$ and $W''(\pm a)>0$. Let $S\subset O(N+1)$ be a subgroup and let $C\subset\R^{N+1}$ be a minimal cone of codimension $1$. Assume that $\Sigma\subset\R^{N+1}$ is an embedded minimal hypersurface such that (H1)-(H4) hold and that $\Sigma$ is $S$-dilation-nondegenerate. Then there exist a family $\{u_\eps\}_{\eps\in(0,\eps_0)}$ of $S$-invariant solutions to equation (\ref{allen-cahn}) in $\R^{N+1}$ 
such that
\begin{enumerate}
\item $u_\eps(\cdotp/\eps)\to\pm a$ as $\eps\to 0$, uniformly on compact subsets of $\Omega_\pm$.
\item $Morse (u_\eps)\ge Morse(\Sigma)$.
\item If in addition, $\Sigma$ is dilation-nondegenerate and $\bar{\nu}\ge -\frac{N-2}{2}$, then $Morse (u_\eps)= Morse(\Sigma)$.
\end{enumerate}
\end{theorem}

\red{}

\begin{remark}
\begin{itemize}
\item We actually have $2$ families of solutions enjoying the properties listed in Theorem \ref{main-th}. In fact, being $W$ even, if $u$ is a solution to (\ref{allen-cahn}), then $-u$ is also a solution.
\item We stress that Theorem \ref{main-th} holds under very weak assumptions about $\Sigma$. In fact, neither stability nor symmetry are required (recall that $S$ may be the trivial group). In particular, our result answers a question raised by Pacard and Wei in Remark $1.1$ of \cite{pacard2013stable} about possible generalisations of their Theorem $1.2$. For further comparison between our work and the results in \cite{pacard2013stable}, we refer to Section $3$.
\end{itemize}
\end{remark}
A particular case of potential fufilling the assumptions of Theorem \ref{main-th} is given by the classical double-well potential $W(u):=\frac{(1-u^2)^2}{4}$, corresponding to the Allen-Cahn equation (\ref{allen-cahn-eq}). In any case, our result holds true for more general potentials. More precisely, it is enough to take a potential which guarantees the existence of an odd heteroclinic solution to the cooresponding ODE connecting the stable equilibria $\pm a$ and converging to $\pm a$ exponentially fast as $t\to\pm\infty$ (see Section $3$).\\

We stress that the assumption that the limit cone $C$ is nontrivial is not restrictive if $W$ is such that the Gibbons conjecture is true, that is $W'<0$ in $(0,a)$, $W'\ge 0$ in $(a,\infty)$ and there exist $0<t_0<t_1<a$ and $\delta_0>0$ such that $W'$ is nondecreasing in $(t_1,a)$ and $W'(t)\le -\delta_0 t$ in $(0,t_0)$ (see \cite{farina2003rigidity}). These conditions are satisfied, for instance, by the classical double-well potential $W(t):=\frac{(1-t^2)^2}{4}$. In fact, in this case any solution $u$ to \eqref{allen-cahn} whose zero level set is bounded in one direction is $1$-dimensional, more precisely $u(\xi)=v_\star(\vec{a}\cdotp(\xi-\xi_0))$, for some $\vec{a}\in \mathbb{S}^N$ and $\xi_0\in\R^{N+1}$, where $v_\star$ is the unique solution to the problem
$$-v''_\star+W'(v_\star)=0,\qquad v_\star(0)=0,\qquad v_\star(\pm\infty)=\pm a.$$
Some explicit examples of minimal hypersurfaces enjoying the hypothesis of Theorem \ref{main-th} are known to exist. As a consequence, we have some remarkable applications of Theorem \ref{main-th}, which will be dealt with in Section $3$. More precisely, we will give an alternative proof of Theorem $1.2$ of \cite{pacard2013stable} about stable solutions to the Allen-Cahn equation (\ref{allen-cahn-eq}) if we start from a stable minimal hypersurface and we will construct new $O(m)\times O(n)$-invariant solutions to (\ref{allen-cahn}) in dimension $4\le m+n=N+1\le 7$ with infinite Morse index.\\

The plan of the paper is the following. In Section $2$ we will consider the Jacobi equation $J_\Sigma\phi=f$ in a given minimal hypersurface $\Sigma$ and give some examples of minimal hypersurfaces which satisfy our nondegeneracy assumptions; in Section $3$ we will discuss some applications of Theorem \ref{main-th} to some particular hypersurfaces; in Section $4$ we perform an ODE analysis which shows that our hypothesis about $W$ are actually enough to guarantee the validity of Theorem \ref{main-th}; in Section $5$ we construct the solutions mentioned in Theorem \ref{main-th} through an infinite dimensional Lyapunov-Schmidt reduction; in Section $6$ we deal with the Morse index 
of such solutions.

\section{Linear theory for the Jacobi equation and examples}

In this section we present a linear theory to solve the Jacobi equation
\begin{equation}
\label{eq-Jacobi}
J_\Sigma\phi=f\qquad\text{in }\Sigma
\end{equation}
on a given minimal hypersurface satisfying (H1)-(H4) under suitable assumptions on $f$. We also recall some examples of $S$-dilation-nondegenerate minimal hypersurfaces. Our discussion here is a briefing of the results and examples presented in \cite{agudelo2025fredholm}. We also refer the reader to \cite{lockhart1985elliptic} and \cite{pacard2008connected} for on hypersurfaces of conic type and connected sums type, respectively.

\vskip 3pt
Following \cite{agudelo2025fredholm}, we introduce the functional analytic setting we use for solving equation \eqref{eq-Jacobi}. Let $\nu \in \R$ and $S$ be a subgroup of $O(N+1)$. For $\alpha \in (0,1)$, the space $C^{0,\alpha}_{\nu}(\Sigma,S)$ denotes the space of $S$-invariant functions $f\in C^{0,\alpha}_{loc}(\Sigma)$ with finite norm
\begin{equation}\notag
\|f\|_{C^{0,\alpha}_{\nu}(\Sigma)}:=\sup_{y\in\Sigma}(|y|+1)^{-\nu}\|f\|_{C^{0,\alpha}(B_1(y))},
\end{equation}
where $B_1(y)\subset \Sigma$ is the geodesic ball in $\Sigma$ of radius one and centered at $y\in \Sigma$.

\vskip 3pt
We also use the notation 
$$\|f\|_{\infty,\nu}:=\|(|y|+1)^{-\nu}f\|_\infty.
$$

The space $C^{2,\alpha}_{\nu}(\Sigma,S)$ denotes the space of $S$-invariant functions $\phi\in C^{2,\alpha}_{loc}(\Sigma)$ with finite norm
\begin{equation}\notag
\|\phi\|_{C^{2,\alpha}_\nu(\Sigma)}:=\sup_{y\in\Sigma}(|y|+1)^{-\nu+2}\|\nabla^2_\Sigma \phi\|_{C^{0,\alpha}(B_1(y))}
+\|\nabla_\Sigma \phi\|_{\infty,\nu-1}+\|\phi\|_{\infty,\nu}.
\end{equation}

We also consider the Jacobi operator of $\Gamma$, $-J_\Gamma:=-(\Delta_\Gamma+|A_\Gamma|^2)$, and we recall that the eigenvalues $\lambda_j$ of $-J_\Gamma:=-(\Delta_\Gamma+|A_\Gamma|^2)$ satisfy 
$$\lambda_0<\lambda_1\le\dots\le \lambda_j\le \dots,\qquad\lambda_j\to\infty\qquad j\to\infty.$$
Recalling the notations in \cite{agudelo2024jacobi}, we set
$$j_0:=\min\left\{j\ge 0:\,\left(\frac{N-2}{2}\right)^2+\lambda_j\ge 0\right\}.$$
For $j\ge j_0$, we set
$$\Lambda_j:=\sqrt{\left(\frac{N-2}{2}\right)^2+\lambda_j}\ge 0.
$$
It is known that $j_0=0$ if and only if the cone $C$ is stable (see \cite{HS}). 
On the other hand, if $C$ is not stable, for 
$0\leq j< j_0$, we set $$\Lambda_j:=i\sqrt{-\left(\left(\frac{N-2}{2}\right)^2+\lambda_j\right)}.$$
The numbers $-\frac{N-2}{2}\pm{\rm Re}(\Lambda_j)$ are known as the \textit{indicial roots} of $C$.\\

\begin{theorem}
\label{prop-right-inverse-Jacobi}
Assume (H1)-(H4). In addition, assume that $\Sigma$ is $S$-dilation-nondegenerate, $\nu>2-N-\bar{\nu}$ and that $\nu$ is not an indicial root of $C$. Then, for any $f\in C^{0,\alpha}_{\nu-2}(\Sigma,S)$, there exist a solution $\phi:=\Phi(f)\in C^{2,\alpha}_{\nu}(\Sigma,S)$ to the Jacobi equation \eqref{eq-Jacobi} such that 
\begin{equation}
\label{est-Jacobi-eq-lin}
\|\phi\|_{C^{2,\alpha}_{\nu}(\Sigma)}\le c\|f\|_{C^{0,\alpha}_{\nu-2}(\Sigma)},  
\end{equation}
for some constant $c>0$ depending only of $\Sigma,\nu$.
\end{theorem}

We refer the reader to Sections 2 and 3 in \cite{agudelo2025fredholm} for the proof of Theorem \ref{prop-right-inverse-Jacobi}.\\

The first collection of examples we discuss is provided by Proposition 2.4 in \cite{agudelo2025fredholm}. To be more precise, under the assumptions (H1)-(H4), if $\Sigma$ is stable and does not intersect the cone $C$, then $\Sigma$ is dilation-nondegenerate. If in addition $\Sigma$ is $S$-invariant, then $\Sigma$ is also $S$-dilation-nondegenerate.

A particular case of the hypersurfaces mentioned above, relevant for the construction of solutions to the Allen-Cahn equation, are next described (see \cite{HS}).

\begin{theorem}[\cite{HS}]
\label{th-Sigma}
Let $N\ge 7$. Assume that $C\subset\R^{N+1}$ is a minimising nontrivial cone. Then there exist exactly two oriented, embedded smooth area minimising hypersurfaces $\Sigma^\pm\subset\R^{N+1}$ such that
\begin{enumerate}
\item $\Sigma^\pm$ are asymptotic to $C$ at infinity and do not intersect $C$.
\item $dist(\Sigma^\pm,\{0\})=1$.
\end{enumerate}
\end{theorem}
Since the hypersurfaces in Theorem \ref{th-Sigma} are area minimising, in particular they are stable.\\

For $N+1=m+n$ and $\xi\in\R^{N+1}$ we write $\xi=(x,y)\in\R^m\times\R^n$ and we consider the Lawson's Cone
$$
C_{m,n}:=
\{(x,y)\in\R^m\times\R^n:\,(n-1)|x|^2=(m-1)|y|^2\},\qquad m,\,n\ge 2.
$$
The Lawson cone is a minimal hypersurface, which is invariant under the action of the group of orthogonal transformations $S=O(m)\times O(n)$. In the case $m=n$, $C_{m,m}$ is known as the Simon's cone.

We remark that for $m,n \geq 2$, the open set $\R^{N+1}\backslash C_{m,n}$ has two connected components, namely
$$
E_{m,n}^\pm:=\{(x,y)\in\R^m\times\R^n:\,\pm((n-1)|x|^2-(m-1)|y|^2)<0\}
$$
corresponding to the interior and the exterior of $C_{m,n}$, respectively. The sets $E_{m,n}^\pm$ help us describe an important feature of $C_{m,n}$ that has been already studied and used in \cite{agudelo2022doubling,agudelo2022k,agudelo2024jacobi,alencar2005m,HS,Mazet_2017} and that is summarised in the next results. 

\begin{theorem}\cite{Mazet_2017}
\label{th-Al}
Let $m,\,n\ge 2$ and set $n+m=N+1\ge 8$. Then there exist exactly two smooth minimal hypersurfaces $\Sigma^\pm_{m,n}\subset E^\pm_{m,n}$ which are asymptotic to $C_{m,n}$ at infinity and ${\rm dist}(\Sigma^\pm_{m,n},0)=1$. Moreover, $\Sigma^\pm_{m,n}$ are $O(m)\times O(n)$-invariant and the Jacobi field $y\cdotp\nu_{\Sigma^\pm_{m,n}}(y)$ never vanishes.
\label{th_min-surf_cone}
\end{theorem}

Considering our explicit examples of minimal hypersurfaces asymptotic to a cone, we have the following result.
\begin{lemma}
\label{lemma-nondeg-examples}
\begin{enumerate}
\item The minimal hypersurfaces constructed in Theorem \ref{th-Sigma}  are dilation-nondegenerate and fulfill $\Lambda_0>0$ and $\bar{\nu}=-\frac{N-2}{2}+\Lambda_0$ if the underlying cone $C$ is strictly minimising, $\bar{\nu}=-\frac{N-2}{2}-\Lambda_0$ if $C$ is minimising but not strictly.
\item The minimal hypersurfaces constructed in Theorem \ref{th-Al}  are dilation-nondegenerate and fulfill $\Lambda_0>0$ and $\bar{\nu}=-\frac{N-2}{2}+\Lambda_0$.
\end{enumerate}
\end{lemma}
\begin{proof}
\begin{enumerate}
\item The dilation-nondegeneracy follows from Proposition 2.4  in \cite{agudelo2025fredholm}, since such hypersurfaces are stable and do not intersect their asymptotic cone $C$. The statement about $\bar{\nu}$ follows from Theorem $3.2$ of \cite{HS}.
\item The dilation-nondegeneracy follows once again from Proposition 2.4 in \cite{agudelo2025fredholm}, exactly as above. For the computation of $\bar{\nu}$ we refer to Remark $5.2$ of \cite{agudelo2024jacobi}.
\end{enumerate}
\end{proof}
In dimension $4\le m+n=N+1\le 7$, we have the following existence result about $O(m)\times O(n)$-invariant minimal hypersurfaces.
\begin{theorem}[\cite{alencar2005m, Mazet_2017}]
\label{th-Sigma-low-dim}
Let $m,n\ge 2$, $m+n\le 7$. Then there exists a unique complete, embedded, $O(m)\times O(n)$ invariant minimal hypersurface $\Sigma_{m,n}$ such that 
\begin{enumerate}
\item \label{Sigma-as} $\Sigma_{m,n}$ is asymptotic to the cone $C_{m,n}$ at infinity,
\item \label{Sigma-osc} $\Sigma_{m,n}$ intersects $C_{m,n}$ infinitely many times,
\item \label{Sigma-ort} $\Sigma_{m,n}$ meets $\R^m\times\{0\}$ orthogonally,
\item \label{Sigma-dist} $dist(\Sigma_{m,n}, C_{m,n})=1$.
\end{enumerate}
\end{theorem}
These hypersurfaces are $S$-dilation-nondegenerate with $S:=O(m)\times O(n)$.
\begin{lemma}\label{lemma-examples-morse-infty}
The minimal hypersurfaces constructed in Theorem \ref{th-Sigma-low-dim} are $O(m)\times O(n)$-dilation-nondegenerate and fulfill ${\rm Re}(\Lambda_0)=0$, $\bar{\nu}=-\frac{N-2}{2}=-\frac{N-2}{2}\pm{\rm Re}(\Lambda_0)$.    
\end{lemma}
\begin{proof}
The result follows from the fact that $\Lambda_0^2=\left(\frac{N-2}{2}\right)^2-(N-1)<0$, since $3\le N\le 6$, and Theorem $5.2$ of \cite{agudelo2024jacobi}.    
\end{proof}

\section{Some applications of Theorem \ref{main-th}}

\subsection{Stable solutions}

In view of our comments in Section 2, Theorem \ref{main-th} gives the following Corollary.
\begin{corollary}\label{cor-stability}
Let $N\ge 7$. Let $W\in C^2(\R)$ be an even function such that $W(\pm a)=0<W(t)$ for any $t\in\R\setminus\{\pm a\}$ and $W''(\pm a)>0$. Let $S\subset O(N+1)$ be a subgroup and let $C\subset\R^{N+1}$ be a minimal cone of codimension $1$. Assume that $\Sigma\subset\R^{N+1}$ is a stable embedded minimal hypersurface such that (H1)-(H4) hold and $\Sigma$ does not intersect $C$. Then there exist a family $\{u_\eps\}_{\eps\in(0,\eps_0)}$ of $S$-invariant solutions to equation (\ref{allen-cahn}) in $\R^{N+1}$ such that
$$u_\eps(\cdotp/\eps)\to\pm a\qquad\text{as $\eps\to 0$,}$$ 
uniformly on compact subsets of $\Omega_\pm$. If, in addition, $\Sigma$ is dilation-nondegenerate and $\bar{\nu}\ge -\frac{N-2}{2}$, our solutions $u_\eps$ are stable.
\end{corollary}
\begin{proof}
Since $\Sigma$ is stable, then by Proposition 2.4 in \cite{agudelo2025fredholm}, it is also $S$-dilation-nondegenerate, hence the existence result directly follows from Theorem \ref{main-th}. If, in addition, $\Sigma$ is dilation-nondegenerate and $\bar{\nu}\ge -\frac{N-2}{2}$, the stability of $u_\eps$ follows from the fact that $Morse(u_\eps)=Morse(\Sigma)=0$.
\end{proof}
Applying Corollary \ref{cor-stability} to the hypersurfaces constructed in Theorem \ref{th-Sigma}, we have the following result.
\begin{corollary}\label{cor-PW}
Let $N\ge 7$. Let $W\in C^2(\R)$ be an even function such that $W(\pm a)=0<W(t)$ for any $t\in\R\setminus\{\pm a\}$ and $W''(\pm a)>0$. Let $C\subset\R^{N+1}$ be a codimension $1$ minimising nontrivial cone and let $\Sigma\subset\R^{N+1}$ be one of the hypersurfaces constructed in Theorem \ref{th-Sigma}. Then there exist a family $\{u_\eps\}_{\eps\in(0,\eps_0)}$ of solutions to equation (\ref{allen-cahn}) in $\R^{N+1}$ 
such that
$$u_\eps(\cdotp/\eps)\to\pm a\qquad\text{as $\eps\to 0$,}$$ 
uniformly on compact subsets of $\Omega_\pm$. If, in addition, $C$ is strictly minimising, our solutions $u_\eps$ are stable.
\end{corollary}
\begin{proof}
Let $N\ge 7$, let $C\subset\R^{N+1}$ be a minimising cone of codimension $1$ and let $\Sigma$ be one of the hypersurfaces constructed in Theorem \ref{th-Sigma}. Then $\Sigma$ is stable, since it is minimising, and it does not intersect the underlying cone $C$. Moreover, let $$S:=\{\rho\in O(N+1):\,\rho(\Sigma)=\Sigma\}.$$ By definition, it follows that $S$ is a subgroup of $O(N+1)$ and (H1)-(H4) are satisfied. Therefore, using Corollary \ref{cor-stability}, the hypersurfaces $\Sigma^\pm$ yield the existence of solutions to equation (\ref{allen-cahn}). By Lemma \ref{lemma-nondeg-examples}, $\Sigma$ is dilation-nondegenerate. Moreover, $\bar{\nu}\ge -\frac{N-2}{2}$ if and only if $C$ is strictly minimising. Therefore, if $C$ is strictly minimising, the corresponding solutions $u_\eps$ are stable.
\end{proof}

\begin{remark}
\begin{itemize}

\item We note that Theorem $1.2$ of \cite{pacard2013stable} is a particular case of Corollary \ref{cor-PW}, which also gives additional information about the stability of the solutions.

\item Moreover, in Theorem $1.2$ of \cite{pacard2013stable} there is the further assumption that $W(u):=\frac{(1-u^2)^2}{4}$, while our potential $W$ is slightly more general.

\item By Lemma $2.1$ of \cite{agudelo2024jacobi}, the assumption $-\frac{N-2}{2}+\Lambda_0>0$, required in Theorem $1.2$ of \cite{pacard2013stable}, is equivalent to the fact that $C$ is nontrivial, which is part of our hypothesis.

\end{itemize}
\end{remark}

\subsection{$O(m)\times O(n)$-invariant solutions in dimension $m+n\ge 8$}
Applying Theorem \ref{main-th} to the Lawson cone $C_{m,n}$ and the hypersurfaces constructed in Theorem \ref{th-Al} with $S:=O(m)\times O(n),\,m,n\ge 2,\,m+n\ge 8$, we have the following Corollary, in which we construct solutions to \eqref{allen-cahn-eq} which are invariant under the action of the group of orthogonal transformations $O(m)\times O(n)$. 

\begin{corollary}\label{cor-cz-highdim}
Let $N+1=m+n\ge 8$, with $m,\,n\ge 2$. Let $W\in C^2(\R)$ be an even function such that $W(\pm a)=0<W(t)$ for any $t\in\R\setminus\{\pm a\}$ and $W''(\pm a)>0$. \color{black}Let $\Sigma:=\Sigma^\pm_{m,n}\subset\R^{N+1}$ be one of minimal hypersurfaces constructed in Theorem \ref{th-Al}. Then there exist a family $\{u_\eps\}_{\eps\in(0,\eps_0)}$ of $O(m)\times O(n)$-invariant stable solutions to equation (\ref{allen-cahn}) in $\R^{N+1}$ 
such that
$$u_\eps(\cdotp/\eps)\to\pm a\qquad\text{as $\eps\to 0$,}$$ 
uniformly on compact subsets of $\Omega_\pm$.  
\end{corollary}
Note that, in Corollary \ref{cor-cz-highdim}, $$\bar{\nu}=-\frac{N-2}{2}+\Lambda_0>-\frac{N-2}{2},$$
which yields that the constructed solutions $u_\eps$ are stable.\\

We remark that the saddle-shaped solutions in \cite{CABRETERRA2009, CABRETERRA2010} enjoy this symmetry too. In the case $m=n$, we conjecture that the solutions $u_{\eps}$ from Corollary \ref{cor-cz-highdim} converge in $L^{\infty}(\R^{N+1})$ to the solutions presented in \cite{CABRETERRA2009}, with convergence in a finer topology away from the singularity of the cone.

If $m+n\ge 8$ and $m,\,n\ge 3$ or $m+n\ge 9$, then the Lawson cone $C_{m,n}$ is minimising, hence the hypersurfaces $\Sigma^\pm_{m,n}$ constructed in Theorem \ref{th-Al} coincide with the ones constructed in Theorem \ref{th-Sigma}, denoted by $\Sigma^\pm$, hence the solutions to equation (\ref{allen-cahn}) constructed in Corollary \ref{cor-cz-highdim} carry more precise information with respect to the ones constructed in Theorem $1.2$ of \cite{pacard2013stable}. Indeed, our solutions here are also $O(m)\times O(n)$-invariant. We conjecture that for $\eps>0$ small enough, these two families of solutions coincide.

The case $m=2$ and $n=6$ or $m=6$ and $n=2$ is not treated in \cite{pacard2013stable} and hence the solutions constructed in Corollary \ref{cor-cz-highdim} are surely new, since in these cases the Lawson cones $C_{m,n}$ are minimal, in the sense that they have $0$ mean curvature everywhere apart from the origin, but they are not minimising.

\subsection{$k$-ended solutions and Morse index}
In the literature many authors are interested in $k$-ended solutions to the Allen-Cahn equation (\ref{allen-cahn-eq}). Roughly speaking, $k$-ended solutions are solutions whose zero level set outside a ball has exactly $k\ge 0$ connected components. A relevant example of $k$-ended solutions in $\R^3$ was constructed in \cite{agudelo2016vyvsvsi}, where each of the connected components of the zero level set is a normal graph over a scaled version of the catenoid. Although the concept of  $k$-ended solutions is widely treated in the literature, at the moment there is no rigorous general definition of it. Here we give our definition. 
\begin{definition}
\label{def-k-ended}
Let $N+1\ge 2$. An entire solution $u$ to (\ref{allen-cahn}) in $\R^{N+1}$ is said to be \textit{$k$-ended} if there exists $R_0>0$ such that, for any $R\ge R_0$, the set $\{x\in\R^{N+1}:\, u(x)=0\}\backslash B_R(0)$ has $k$ connected components.
\end{definition}

\vskip 5pt
In \cite{agudelo2022doubling,agudelo2022k} the authors consider one of the hypersurfaces $\Sigma:=\Sigma^\pm_{m,n}$ constructed in Theorem \ref{th-Al} and construct $k$-ended solutions to \eqref{allen-cahn-eq} which change sign 
$k$ times near a large dilation of $\Sigma$, for any $k\ge 2$.

\begin{theorem}\cite{agudelo2022k}
\label{2ndMainTheo}
Let {$m,\,n\ge 2$}, $n+m=N+1\ge 8$ and $k\ge 2$. 
Then there exists $\eps_0>0$ such that, for any $\eps\in(0,\eps_0)$, there exists a $k$-ended  solution $u_\eps$ to \eqref{allen-cahn-eq} in $\R^{N+1}$ such that 
\begin{enumerate}
\item[(i)] $u_\eps$ is smooth and $O(m)\times O(n)$-invariant;


\item[(ii)] \label{morse-infty} $Morse(u_\eps)=\infty$.

\item[(iii)] \label{est_energy_grad} there exists a constant $c>0$ such that for any $\eps \in (0,\eps_0)$ and any $R> 2\eps^{-1}$,
\begin{equation}\notag
\int_{B_R}\frac{1}{2} |\nabla u_\eps|^2+\frac{1}{4}(1-u_\eps^2)^2\le c R^N.
\end{equation}
\end{enumerate}
\end{theorem}

A few comments are now in order. First, from Theorem \ref{2ndMainTheo} and taking $m,\,n \geq 2$, there are two associated minimal hypersurfaces $\Sigma^{\pm}_{m,n}$, each of which giving rise to a family of solutions. Since the nonlinearity is odd, if $u$ is a solution then $-u$ is also a solution. Thus, we have actually $4$ families of solutions.

$$
$$
$$
$$




On the other hand, part {\rm (iii)} in Theorem \ref{2ndMainTheo} implies that for any $\eps \in (0,\eps_0)$, 
$$
\limsup \limits_{R \to \infty} \frac{1}{R^N}\int_{B_R}\frac{1}{2}|\nabla u_\eps|^2 +\frac{1}{4} (1-u^2_\eps)^2< \infty
$$
and this suggests that these solutions should have {\it finite Morse index}, see \cite{{ambrosio2000entire},{ghoussoub1998conjecture}}. In any case, this is not true, which disproves a conjecture raised by Del Pino, Kowalczyk and Wei in \cite{del2013entire}.

Theorem \ref{2ndMainTheo} provide examples of $k$-ended solutions with infinite Morse index. A natural question is whether given a solution to the Allen-Cahn equation (\ref{allen-cahn-eq}) having infinite Morse index, then for some $k\in \mathbb{N}$, $k\geq 2$, this solution is  $k$-ended. In the next subsection we will see that it is not the case, at least in dimension $4\le N+1\le 7$.

\subsection{$O(m)\times O(n)$-solutions in dimension $4\le m+n\le 7$}

In dimension $4\le m+n=N+1\le 7$, the hypersurfaces constructed in Theorem \ref{th-Sigma-low-dim} can be used to construct $O(m)\times O(n)$-invariant solutions to \eqref{allen-cahn}. More precisely, we have the following Corollary of Theorem \ref{main-th}.
\begin{corollary}\label{cor-cz-lowdim}
Let $m,\,n\ge 2$, $N+1=m+n\le 7$. Let $W\in C^2(\R)$ be an even function such that $W(\pm a)=0<W(t)$ for any $t\in\R\setminus\{\pm a\}$ and $W''(\pm a)>0$. Let $\Sigma:=\Sigma_{m,n}\subset\R^{N+1}$ be the minimal hypersurface constructed in Theorem \ref{th-Sigma-low-dim}. Then there exist a family $\{u_\eps\}_{\eps\in(0,\eps_0)}$ of $O(m)\times O(n)$-invariant solutions to equation (\ref{allen-cahn}) in $\R^{N+1}$ 
such that
$$u_\eps(\cdotp/\eps)\to\pm a\qquad\text{as $\eps\to 0$,}$$ 
uniformly on compact subsets of $\Omega_\pm$ and ${\rm Morse}(u_\eps)=\infty$.   
\end{corollary}
\begin{proof}
It is enough to apply Theorem \ref{main-th} with $\Sigma:=\Sigma_{m,n}$ and $S=O(m)\times O(n)$, since $\Sigma_{m,n}$ is $O(m)\times O(n)$-dilation-nondegenerate. Moreover, it was proved in Theorem $1.5$ of \cite{agudelo2024jacobi} that ${\rm Morse}(\Sigma)=\infty$, therefore, by point $(2)$ of Theorem \ref{main-th}, ${\rm Morse}(u_\eps)=\infty$. We note that, in this case, the inequality $Morse(u_\eps)\ge Morse (\Sigma)$ is enough, hence we do not need to know if $\Sigma$ is dilation-nondegenrate.
\end{proof}
\begin{remark}
Corollary \ref{cor-cz-lowdim} is completely new and shows an example of solutions to (\ref{allen-cahn}) which have infinite Morse index but are not $k$-ended.    
\end{remark}

\section{Analysis of heteroclinic solutions - ODE analysis }

In this part we discuss existence, symmetry and asymptotic behavior of a heteroclinic solution to the one dimensional equation 
\begin{equation}\label{ODE-vstar}
-v''_\star +W(v_\star)=0 \quad \hbox{in }\R, \qquad v_\star(0)=0, \qquad v_\star(\pm \infty)=\pm a.
\end{equation}

Assume that 
$W\in C^2(\R)$, $W(\pm a)=0$, $W(s)>0$ for $s\neq \pm a$, and $W''(\pm a)>0$.

Consider the Cauchy problem in $(0,\infty)$ for the differential equation in \eqref{ODE-vstar} with the initial conditions $v_\star(0)=0$ and $v'_{\star}(0)=\sqrt{2W(0)}$. Existence of the local solution $v_\star$ for this problem is guaranteed by the classic ODE theory. With the use of the energy of $v_\star$, namely the function
$$
t\mapsto E(v_\star)(t)= -\frac{1}{2}\big(v'_\star(t)\big)^2 + W(v_\star(t)),
$$
it is proven that $v_\star, v'_\star \in L^{\infty}(0,\infty)$ with $0<v_\star
(t)< a$ for $t>0$. This, together with the fact that 
$$
v'_\star(t)=\sqrt{2 W\big(v_{\star}(t)\big)} >0 \quad \hbox{for }t>0, 
$$
yield that $v_{\star}(\infty)=a$. Proceeding in the same manner the respective Cauchy problem in $(-\infty,0)$ is solved and hence a monotone heteroclinic solution $v_\star:\R \to (-a,a)$ exists.

\begin{remark}\label{rem-symmetry-v-star}
    We note that, if $W$ is even, then the phase plane shows that $v_\star$ is odd.
\end{remark}

Consider the linearized operator $L:=-\frac{d^2}{dt^2} + W''(v_{\star}(t))$. The monotonicity of  $v_\star$, i.e., $v'_\star>0$, yields its stability, that is the operator $L$ satisfies maximum principle or equivalently,
$$
\int_{-
\infty}^{\infty}vL(v)dt \in [0,\infty) \quad \hbox{for any} \quad v\in H^1(\R).
$$

Next, set $\gamma_{\pm}:=\sqrt{W''\big(\pm a\big)}$. Applying the comparison principle for the linearized operator $L=-\frac{d^2}{dt^2} + W''(v_{\star}(t))
$
we observe that for any $k\in \mathbb{N}$, there exists $c_k>0$ such that
\begin{equation}
\label{as-vstar+infty}
\left|\frac{d^k}{dt^k}(v_\star-a)\right|\le c_k e^{-\gamma_+ t} \qquad\text{for $t>0$}
\end{equation}
and
\begin{equation}
\label{as-vstar-infty}
\left|\frac{d^k}{dt^k}(v_\star+a)\right|\le c_k e^{\gamma_- t} \qquad\text{for $t<0$.}
\end{equation}

Finally, we note that $v'_\star\in H^1(\R)$ and is such that $Lv'_{\star}=0$. On the other hand, the function
$$w(t):=v'_\star(t)\int_0^t (v'_\star(s))^{-2}ds$$
defines a second linearly independent solution to the ODE $Lv=0$ whose asymptotic behaviour is given by $w=O(e^{\gamma_\pm|t|})$ as $t\to\pm\infty$, respectively.\\

As a consequence, we get the following result.
\begin{lema}
\label{lema-nonde-vstar}
Let $v\in H^1(\R)$ be a solution to $Lv=0$. Then $v=cv'_\star$, for some constant $c\in\R$.
\end{lema}
\begin{proof}
Since $v\in H^1(\R)$, then $v\in C^2\cap L^\infty(\R)$ is a classical solution to $Lv=0$, so that $v=c_1 v'_\star+c_2 w$. Since $v$ is bounded, we have $c_2=0$.
\end{proof}

We will need the following non-degeneracy result.
\begin{lema}
There exists a constant $\lambda>0$ such that any $\varphi\in H^1(\R)$ with $\int_\R \varphi v'_\star dt=0$ satisfies that
$$
\int_\R \Big[(\varphi')^2+W''(v_\star)\varphi^2\Big]dt \ge \lambda \int_\R \varphi^2dt.
$$
\end{lema}

\begin{proof}
Set 
\begin{equation}\label{def:lambda_linear_heteroclinic}
\lambda:=\inf\left\{\frac{\int_\R \Big[(\varphi')^2+W''(v_\star)\varphi^2\Big]dt }{\int_\R \varphi^2 dt}:\,\varphi\in H^1(\R),\,\int_\R \varphi v'_\star dt =0\right\}\ge 0.\end{equation}

Since $v_\star$ is stable, $\lambda \geq 0$. It is enough to prove that $\lambda>0$. 

Let $\varphi_k\in H^1(\R)$ be a minimising sequence for \eqref{def:lambda_linear_heteroclinic}, that is, 
\begin{equation}
\label{prop-varphik}
\int_\R \varphi_k v'_\star dt =0,\,\int_\R \varphi_k^2 dt=1,\,\int_\R \Big[(\varphi_k')^2+W''(v_\star)\varphi_k^2 \Big]dt \to \lambda.
\end{equation}
as $k\to\infty$. 

Since $W''(v_\star)\in L^{\infty}(\R)$, it follows from \eqref{prop-varphik} that $\varphi_k$ is bounded in $H^1(\R)$.

By passing to a subsequence, still denoted by $\{\varphi_k\}_k$, we assume that $\varphi_k\rightharpoonup \varphi$ weakly in $H^1(\R)$ and $\varphi_k\to \varphi$ strongly in $L^2_{loc}(\R)$.\\

We first assume that there exists a subsequence $\{\varphi_{k_j}\}_j$ such that for every $j\in \mathbb{N}$
$$
\int_{|t|\leq j} \varphi^2_{k_j}dt < \frac{1}{j}.
$$

Under this assumption, we proceed as follows. Using this subsequence, there exist $c_1,c_2>0$ and $j_0\in \mathbb{N}$ such that for $j\geq j_0$,
\begin{equation}
\label{ineq:estima1}
\begin{aligned}
\int_\R(\varphi_{k_j}')^2+W''(v_\star)\varphi_{k_j}^2 dt&\ge \int_\R W''(v_\star)\varphi_{k_j}^2dt\\
&=\int_{|t|<j} W''(v_\star)\varphi_{k_j}^2+\int_{|t|\ge j} W''(v_\star)\varphi_{k_j}^2\\
&\ge -\frac{c_2}{j}+c_1(1-\frac{1}{j})\ge \frac{c_1}{2}>0.
\end{aligned}
\end{equation}

From \eqref{prop-varphik} and \eqref{ineq:estima1}, taking $j\to \infty$, we obtain that $\lambda>0$.

Let us now assume that there exist $\epsilon,\,R,\,k_0>0$ such that
\begin{equation}
\label{varphi_k-concentrated}
\int_{|t|<R} \varphi_k^2 dt\ge \epsilon,\quad \hbox{for every }k\ge k_0.
\end{equation}

Let us prove that also in this case, $\lambda>0$. We argue by contradiction, i.e., assume $\lambda=0$.

Due to the strong convergence in $L^2_{loc}(\R)$, we have
\begin{equation}
\label{ineq:varphi_not_zero}  
\int_{|t|<R}\varphi_k^2\to\int_{|t|<R}\varphi^2\ge \epsilon>0
\end{equation}
so that $\varphi\ne 0$.\\

By the weak lower semicontinuity of the norm and the Fatou Lemma, we have
$$\int_\R (\varphi')^2\le \liminf_{k\to\infty}\int_\R (\varphi')^2\qquad\int_{|t|>R}W''(v_\star)\varphi^2\le \liminf_{k\to\infty}\int_{|t|>R}W''(v_\star)\varphi_k^2.$$
Finally, using stability we have 
$$0\le\int_\R (\varphi')^2+W''(v_\star)\varphi^2\le \liminf_{k\to\infty}\int_\R (\varphi_k')^2+W''(v_\star)\varphi_k^2=\lambda=0.$$

Therefore $\varphi\in H^1(\R)\setminus\{0\}$ is a minimiser of \eqref{def:lambda_linear_heteroclinic}. The associated {\it Euler-Lagrange} equation reads as $L\varphi =0$ in $\R$. Due to Lemma \ref{lema-nonde-vstar}$,\varphi=cv'_\star$ for some $c\in \R$. However, passing to the limit in the orthogonality condition in (\ref{prop-varphik}), we can see that $\varphi=0$, which contradicts \eqref{ineq:varphi_not_zero}. This completes the proof.
\end{proof}

\section{Construction of solutions}

In this Section we prove point (1) of Theorem \ref{main-th}, that is we construct a family $\{u_\eps\}_{\eps\in(0,\eps_0)}$ of solutions to equation (\ref{allen-cahn}) whose zero level set is close to a scaled copy $\Sigma_\eps:=\eps^{-1}\Sigma$ of a given minimal hypersurfaces $\Sigma$, which is only assumed to be $S$-invariant, asymptotic to a given minimal cone $C$ and $S$-dilation-nondegenerate. We stress that $\Sigma$ and $C$ do not need to 
satisfy neither stability nor symmetry properties, since $S$ can be the trivial group.

\subsection{The Fermi coordinates}
First, we recall that the principal curvatures of the cone $C$ are $O(1/|y|)$ as $|y|\to \infty$ along $C$. In view of hypothesis (H2), and using for instance a Fourier expansion in $\theta$ of the function $w(r,\theta)$, which defines $\Sigma$ as a normal graph over $C$, the principal curvatures $k_i$ of $\Sigma$ have this same property, i.e., there exists $c>0$ such that 
\begin{equation}
\label{est-princ-curv}
|\kappa_i(y)|\le \frac{c}{1+|y|}\qquad\forall\, y\in\Sigma.
\end{equation}

Next, let $\nu_{\Sigma}:\Sigma \to \mathbb{S}^{N}$ be the choice of the unit normal vector of $\Sigma$ pointing in the same direction as $\nu_C$, the normal vector to the cone $C$.

We introduce the Fermi coordinates in a tubular neighborhood of $\Sigma$. More precisely, consider the the mapping
\begin{equation}\label{def:Fermi_coord}
X:\Sigma\times\R\to\R^{N+1}, \qquad 
X(y,z):=y+z\nu_\Sigma(y).
\end{equation}

It is known that there exist $\delta>0$ such that for any $\eta>0$ the mapping in \eqref{def:Fermi_coord} defines a change of coordinates, known as the {\it Fermi coordinates} around $\Sigma$, in the tubular neighborhood
$$\mathcal{N}:=\left\{y+z\nu_\Sigma(y),\,y\in\Sigma,\,|z|<\delta(1+\eta\log(|y|+2))\right\}\subset\R^{N+1}.$$

The comments above guarantee that the amplitude of the neighborhood of $\Sigma$ where the  Fermi coordinates define a smooth diffeomorphism is linearly increasing with $|y|$. In particular, these coordinates are indeed well defined in $\mathcal{N}$. This will be enough for our purposes.

Next, we calculate the Laplace operator in {\it shifted} Fermi coordinates. We follow the strategy and calculations found in Section $4$ of \cite{agudelo2022doubling}. For $z$ fixed, we denote the normal graph over $\Sigma$ of the constant function $z$ by
$$\Sigma^z:=\{y+z\nu_\Sigma(y):\,y\in\Sigma,\,X(y,z)\in\mathcal{N}\}.$$

Using the fact that $H_\Sigma=0$ we expand
\begin{equation}\label{for:mean_curv_Sigmaz}
H_{\Sigma_z}=\sum_{i=1}^N \frac{\kappa_i(y)}{1-z\kappa_i(y)}
\end{equation}
(see also formula ($3.20$) of \cite{del2011giorgi}). \\













For $z\in\R$ fixed, we have the following result.
\begin{lemma} Let $H_{\Sigma^z}$ and $\Delta_{\Sigma^z}$ be the mean curvature and the Laplace-Beltrami operator of $\Sigma_z$, respectively. There exist smooth functions ${\tt Q}(y,z)$ and $a^{ij}(y,z)$ and $b^i(y,z)$ such that
\begin{equation}\notag
\begin{aligned}
H_{\Sigma^z}&=z|A_\Sigma|^2+{\tt Q}(y,z)\\
\Delta_{\Sigma^z}&=\Delta_\Sigma+a^{ij}(y,z)\partial_{ij}+b^i(y,z)\partial_i,
\end{aligned}
\end{equation}
where summation is taken over repeated indexes and the coefficients $a^{ij}$ and $b^{i}$ fulfil the estimates
\begin{equation}
\label{est-remainder-lapl}
|{\tt Q}(y,z)|\le cz^2(1+|y|)^{-3},\,
|a^{ij}(y,z)|\le c|z|(1+|y|)^{-1},\,
|b^{i}(y,z)|\le c|z|(1+|y|)^{-2}
\end{equation}
for any $(y,z)\in X^{-1}(\mathcal{N})$.
\end{lemma}
\begin{proof} First we deal with the calculations for $H_{\Sigma^z}$. Set 
\begin{equation}\notag
{\tt Q}(y,z):=H_{\Sigma_z}-z|A_\Sigma|^2
\end{equation}
and consider the functions  $f_i(y,z):=\frac{\kappa_i(y)}{1-z\kappa_i(y)}$ for $(y,z)\in \mathcal{N}$. Observe that, for any $(y,z)\in X^{-1}(\mathcal{N})$ and $l\ge 0$,
\begin{equation}\label{ineq:Q}
|\partial_z^{(l)}f_i(y,z)|=\left|l!\frac{\kappa_i(y)^{l+1}}{(1-z\kappa_i(y))^{l+1}}\right|
\le  c_l|\kappa_i^{l+1}(y)|
\le c_l(1+|y|)^{-(l+1)}.
\end{equation}

These estimates together with \eqref{for:mean_curv_Sigmaz} imply that
\begin{equation}
\label{est-H_Sigma_z}
|{\tt Q}(y,z)|\le cz^2(1+|y|)^{-3}\qquad\forall\,(y,z)\in X^{-1}(\mathcal{N})
\end{equation}
and this relation can be differentiated in both variables.

\vskip 3pt
Next, for the calculations of $\Delta_{\Sigma^z}$ we proceed as follows. Let $Y:\Omega \subset \R^N\to \Sigma$ be a local parametrisation of $\Sigma$, where $\Omega$ is an open set. Let $\Pi:\Sigma\times\R\to\Sigma$ denote the projection onto $\Sigma$, i.e., $\Pi(y,z):=y$ for any $(y,z)\in \Sigma\times \R$.

Recall that $z\in \R$ is arbitrary, but fixed. Consider the open set
$$
\Omega':= \Big(X \circ \big( Y(\cdot),z\big)\Big)^{-1}(\mathcal{N})=Y^{-1}(\Pi(X^{-1}(\Sigma^z)))\subset\Omega
$$
and the mapping 
$$
Y_z:\Omega'\mapsto \Sigma^z, \qquad Y_z(\x) := X(Y(\x),z)
$$
 so that $Y_z$ is a parametrisation of $\Sigma^z$.

Let $g=(g_{ij})_{N\times N}$ be the metric induced by the standard metric of $\R^{N+1}$ on $Y(\Omega)\subset \Sigma$. The coefficients of the corresponding induced metric on $Y_z(\Omega')$ in local coordinates read as 
\begin{equation}
\label{exp-G_ij}
G_{ij}=g_{ij}-2zA_{ij}+z^2 N_{ij},\qquad 1\le i,j\le N,
\end{equation} 
where $A_{ij}$ are the coefficients of the second fundamental form of $\Sigma$ and $N_{ij}=\partial_i\nu_{\Sigma}\cdotp\partial_j\nu_\Sigma$.
The Laplace-Beltrami operator $\Delta_{\Sigma^z}$ is then calculated from the well-known formula
\begin{equation}
\Delta_{\Sigma_z}=\frac{1}{\sqrt{|G|}}\partial_i(\sqrt{|G|}G^{ij}\partial_j)=G^{ij}\partial_{ij}+(\partial_i G^{ij}+\bar{\Gamma}_{ki}^k G^{ij})\partial_j,
\label{lapl-Sigma_z}
\end{equation}
where $|G|:=|\det G|$, $G^{-1}=(G^{ij})_{N\times N}$ is the inverse of the metric $G$ and 
\begin{equation}
\label{Chr-symbols}
\bar{\Gamma}^i_{jk}=\frac{1}{2}G^{im}\left(\frac{\partial G_{mk}}{\partial\x^j}+\frac{\partial G_{mj}}{\partial\x^k}
-\frac{\partial G_{kj}}{\partial\x^m}\right)
\end{equation}
are the Christoffel symbols of $\Sigma^z$.\\

Similar calculations as those in \cite{agudelo2022doubling} yield
\begin{equation}\notag
\Delta_{\Sigma^z}=\Delta_{\Sigma}+a^{ij}(y,z)\partial_{ij}+b^i(y,z)\partial_i,
\end{equation}
where
\begin{equation}\notag
\begin{aligned}
a^{ij}&:=G^{ij}-g^{ij},\qquad 1\le i,j\le N\\
b^j&:=\partial_i G^{ij}+\bar{\Gamma}_{ki}^k G^{ij}-\partial_i g^{ij}-\Gamma_{ki}^k g^{ij},\qquad 1\le i,j\le N
\end{aligned}
\end{equation}
where $\Gamma^i_{jk}$ are the Christoffel symbols of $\Sigma$.

Recall the matrices $A=(A_{ij})_{N\times N}$ and $N=(N_{ij})_{N\times N}$. Consider the symmetric matrices $g^{-1}A:=(A^i_j)_{N\times N}$ and $g^{-1}N:=(N^i_j)_{N\times N}$, i.e., $A^{i}_j:=g^{ik}A_{kj}$ and $N^{i}_j:=g^{ik}N_{kj}$ for $1\leq i,j\leq N$. As a consequence, (\ref{exp-G_ij}) reads as 
$$G_{ij}=g_{ik}(\delta_j^k-2zA_j^k+z^2 N_j^k)$$

Let us estimate the coefficients $A^i_j$. Since the principal curvatures $\kappa_i$ of $\Sigma$ are the eigenvalues of $g^{-1}A$, for any $y\in\Sigma$, there exists an orthonormal basis $\{v_1(y),\dots,v_N(y)\}$ of eigenvectors of $g^{-1}A$ such that 
$$
g^{-1}A=Q^{-1}{\rm diag}(\kappa_1(y),\dots,\kappa_N(y))Q,
$$ 
where $Q$ is the orthogonal matrix function whose columns are given by $v_1(y),\dots,v_N(y)$. Since the matrix functions $Q(y)$ and $Q^{-1}(y)$ are uniformly bounded, from \eqref{est-princ-curv} we conclude that
\begin{equation}
\label{decayA^i_j}
|A^i_j(y)|\le c(1+|y|)^{-1}\qquad\forall\,1\le i,j\le N.
\end{equation}

Next, to estimate the coefficients $N^i_j$, we use again hypothesis (H3) to find that,
\begin{equation}\notag
    |N^{i}_j(y)|\le c(1+|y|)^{-2}\qquad\forall\,1\le i,j\le N.
\end{equation}
the inverse $G^{-1}$ of the metric on $\mathcal{N}$ fulfils
\begin{equation}
\label{exp-G^ij}
G^{ij}=g^{ik}(\delta_k^j+a_k^j(y,z))=g^{ij}+a^{ij}(y,z)
\end{equation}
with $a^{ij}$ fulfilling
\begin{equation}
\label{est-bar-a^ij}
|a^{ij}(y,z)|\le c|z|(1+|y|)^{-1}\qquad\forall\,(y,z)\in X^{-1}(\mathcal{N}),
\end{equation}
and this estimate can be differentiated in $y$.\\

In order to estimate $b^j$ we note that 
by (\ref{exp-G_ij}) and (\ref{exp-G^ij})
$$|\bar{\Gamma}^{k}_{ki}-\Gamma^{k}_{ki}|\le c|z|(1+|y|)^{-2}\qquad\forall\,(y,z)\in X^{-1}(\mathcal{N}),\,1\le i\le N,$$
due to (\ref{est-bar-a^ij}). As a consequence
\begin{equation}
\label{est-b^j}
\begin{aligned}
b^j&=\partial_i (G^{ij}-g^{ij})+(\bar{\Gamma}^k_{ki}-\Gamma^k_{ki})G^{ij}+\Gamma^k_{ki}(G^{ij}-g^{ij})=\\
&\partial_i a^{ij}+\Gamma^k_{ki}a^{ij}+(\bar{\Gamma}^k_{ki}-\Gamma^k_{ki})G^{ij},
\end{aligned}
\end{equation}
so that
\begin{equation}\notag
|b^j(y,z)|\le c|z|(1+|y|)^{-2}\qquad\forall\,(y,z)\in X^{-1}(\mathcal{N}).
\end{equation}
\end{proof}
Given a shift function $h\in C^2(\Sigma)\cap W^{2,\infty}(\Sigma)$ with $\|h\|_{L^\infty(\Sigma)}<\frac{\delta}{3}$, we introduce a new change of variables by setting 
\begin{equation}\label{shifted-coord}
z:=\eps(t+h(\eps \y)).    
\end{equation}
In other words, the mapping $$X_{\eps,h}(\y,t):=\y+(t+h(\eps \y))\nu_\Sigma(\eps \y)$$
defines new coordinates $(\y,t)$ in the tubular neighbourhood $$\mathcal{N}_{\eps,h}:=\left\{\y+(t+h(\eps \y))\nu_\Sigma(\eps \y):\,\y\in\Sigma_\eps,\,|t|<\frac{\delta}{2\eps}+\frac{\delta\tilde{\eta}}{\eps}\log(|\eps\y|+2)\right\}$$
of the scaled hypersurface $\Sigma_\eps:=\eps^{-1}\Sigma$, known as the \textit{shifted Fermi coordinates}. The constant $\tilde{\eta}>0$ will be determined in the proof of Lemma \ref{lemma-error-far}.\\

We are interested in expanding the Laplacian in shifted Fermi coordinates. The computations can be found in Section $4$ of \cite{agudelo2022doubling}, they are identical, here we only recall the expansion
\begin{equation}
\label{Laplacian_shifted_Fermi}
\begin{aligned}
\Delta&=\Delta_{\Sigma_\eps}+\partial^2_t-\eps^2\big(\Delta_\Sigma h+|A_\Sigma|^2 h\big)\partial_t- \eps^2 t|A_{\Sigma}|^2 \partial_{t} - 2 \eps\nabla_\Sigma h\cdotp\partial_t\nabla_{\Sigma_\eps}+\eps^2|\nabla_\Sigma h|^2\partial^2_t\\
&-\eps{\tt Q}\partial_t +(a^{ij}\partial^2_{ij}+\eps b^j\partial_j)-\eps^2(a^{ij} \partial_{ij}h+b^j \partial_j h)\partial_t-2\eps a^{ij} \partial_i h\partial_{tj}+\eps^2a^{ij}\partial_i h\partial_j h \partial^2_t,
\end{aligned}
\end{equation}
where $1\le i,j\le N$, $\tt h$ and its derivatives are evaluated at $\eps \y$, while ${\tt Q},\, a^{ij}$ and $b^j$ are evaluated at $(\eps \y,\eps(t+h(\eps \y)))$. We recall that $a^{ij}$ and $b^i$ satisfy (\ref{est-remainder-lapl}). Here the remainders $\bar{a}^{ij}$, $\bar{b}^i$ and ${\tt Q}$ fulfil the estimates
\begin{equation}
\label{est_remainder_laplacian}
\begin{aligned}
|{\tt Q}(\eps \y,\eps(t+h(\eps \y)))|&\le C\frac{\eps^2(t+h(\eps \y))^2}{(1+\eps |\y|)^3}\\
|a^{ij}(\eps \y,\eps(t+h(\eps \y)))|&\le C\frac{\eps|t+h(\eps \y)|}{(1+\eps |\y|)}
\\
|b^i(\eps \y,\eps(t+h(\eps \y)))|&\le C\frac{\eps|t+h(\eps \y)|}{(1+\eps |\y|)^2}.
\end{aligned}
\end{equation}
\subsection{The approximate solution}
We take a shift function $h\in C^{2,\alpha}_\nu(\Sigma)$, for some suitable $\nu\in\R$, and we define the first approximation of our solution to \eqref{allen-cahn} as $v_\star(t)$, where $(t,\y)$ are the shifted Fermi coordinates defined in \eqref{shifted-coord}. In order to choose the appropriate $\nu$, we recall that
\begin{equation}\notag
\bar{\nu}:=\inf\{\nu\in\R:\, \limsup_{|y|\to\infty}|y|^{-\nu}|\zeta_0(y)|=0\}
\end{equation}
is the decay rate of the polynomial part of the Jacobi field $\zeta_0(y):=y\cdotp \nu_\Sigma(y)$. Recalling that $\Sigma$ is $S$-dilation-nondegenerate, we have $\bar{\nu}>2-N$, hence it is possible to take 
\begin{equation}
\label{choice-nu}
\nu\in\left(\left(2-N-\bar{\nu},0\right)\cap[-1,0)\right)\setminus\{-\frac{N-2}{2}\pm{\rm Re}(\Lambda_j)\}_{j\ge 0}    
\end{equation}
This is actually the suitable choice to solve \eqref{allen-cahn}.


The reasons of the restrictions about $\nu$ are the following: we need to take $\nu<0$ since we want $h$ to decay at infinity, so that the zero level set of our solutions will be asymptotic to $\Sigma$ at infinity. Moreover, we need to take $\nu\ge -1$ since the error coming from the approximation will decay as $(1+|\eps\y|)^{-3}$ (see the estimate of $E_2$ in \eqref{est-error} below), hence we want $(1+|y|)^{-3}$ to belong to $C^{0,\alpha}_{\nu-2}(\Sigma)$, which is the case if we require $\nu-2\ge-3$. The other restrictions come from Theorem \ref{prop-right-inverse-Jacobi}, which guarantees the existence of a right inverse of the Jacobi operator $J_\Sigma$ if $\nu>2-N-\bar{\nu}$ and $\nu+\frac{N-2}{2}\notin \{\pm{\rm Re}(\Lambda_j)\}_{j\ge 0}$. \\

As a consequence, given a constant $C_1>0$ (to be determined later through a fixed point argument), a constant $\alpha\in(0,1)$ and a shift function $h\in C^{2,\alpha}_\nu(\Sigma)$ such that $\|h\|_{C^{2,\alpha}_\nu(\Sigma)}<C_1\eps$, a first approximation can be given, in shifted Fermi coordinates, by $v_\star(t)$, where $v_\star$ solves (\ref{ODE-vstar}). Using the expansion of the Laplacian in shifted Fermi coordinates (\ref{Laplacian_shifted_Fermi}) and the ODE satisfied by $v_\star$, it is possible to see that
\begin{equation}
\label{first-error}
-\Delta v_\star+W'(v_\star)=\eps^2 J_\Sigma h v'_\star+\eps^2|A_\Sigma|^2 tv'_\star-\eps^2|\nabla_\Sigma h|^2 v''_\star\\
+\eps{\tt Q}v'_\star+\eps^2(a^{ij}\partial_{ij}h+b^i\partial_i h)v'_\star-\eps^2 a^{ij}\partial_i h\partial_j h v''_\star
\end{equation}
where $|A_\Sigma|^2$, $h$ and its derivatives are evaluated at $\eps\y$ while ${\tt Q}$, $a^{ij}$ and $b^i$ are evaluated at $(\eps\y,\eps(t+h(\eps\y)))$. The term $\eps^2|A_\Sigma|^2 tv'_\star$ in (\ref{first-error}) is too large for our purposes, since $|A_\Sigma(y)|^2=O(|y|^{-2})$ as $|y|\to\infty$ only, therefore we add a correction of the form $\eps^2|A_\Sigma(\eps\y)|^2\eta(t)$, where $\eta:\R\to\R$ solves the problem
$$-\eta''+W''(v_\star)\eta=-tv'_\star\qquad\int_\R \eta v'_\star=0,$$ 
so that the new approximated solution $U(\y,t):=v_\star(t)+\eps^2|A_\Sigma(\eps\y)|^2\eta(t)$ is such that
\begin{equation}
\label{improved-error}
-\Delta U+W'(U)=\eps^2 J_\Sigma h v'_\star+E_1(\y,t)+E_2(\y,t),
\end{equation}
where the error terms $E_1$ and $E_2$ fulfil
\begin{equation}
\label{est-error}
\begin{aligned}
&\int_\R E_1(\y,t) v'_\star(t)dt=0\qquad\forall\,\y\in \Sigma_\eps\\
&|E_1(\y,t)|\le c\eps^4 e^{-\gamma |t|}(1+|\eps\y|)^{2\nu},\,|E_2(\y,t)|\le c\eps^3 e^{-\gamma |t|}(1+|\eps\y|)^{-3}\qquad\forall\,(\y,t)\in X_{\eps,h}^{-1}(\mathcal{N}_{\eps,h})
\end{aligned}
\end{equation}
for some $\gamma\in(0,\sqrt{W''(a)})$. The proof of (\ref{improved-error}) and (\ref{est-error}) relies once again on the expansion of the Laplacian in shifted Fermi coordinates (\ref{Laplacian_shifted_Fermi}) and the choice of $h$. In fact, an explicit computation shows that
\begin{equation}
\label{def-E1}
\begin{aligned}
E_1(\y,t)=&-\eps^2|\nabla_\Sigma h|^2 v''_\star-\eps^2 a^{ij}h_i h_j v''_\star-\eps^4\Delta_\Sigma |A_\Sigma|^2\eta+\eps^4 |A_\Sigma|^4 t\eta'-\\
&\eps^4|\nabla_\Sigma h|^2 |A_\Sigma|^2 \eta''-\eps^4(a^{ij}\partial_{ij}|A_\Sigma|^2+b^i\partial_i|A_\Sigma|^2)\eta-\eps^4 a^{ij}h_i h_j |A_\Sigma|^2\eta''
\end{aligned}
\end{equation}
and
\begin{equation}
\label{def-E2}
\begin{aligned}
E_2(\y,t)=&W'(U)-W'(v_\star)-\eps^2|A_\Sigma|^2 W''(v_\star)\eta+\eps {\tt Q}v'_\star+\eps^2\left(a^{ij}h_{ij}+b^jh_j\right) v'_\star\\
&+\eps^4 J_\Sigma h |A_\Sigma|^2 \eta'+2\eps^4 \nabla_\Sigma h\cdotp\nabla_\Sigma |A_\Sigma|^2 \eta'+\eps^3{\tt Q}|A_\Sigma|^2\eta'\\
&+\eps^4(a^{ij}h_{ij}+b^i h_i)|A_\Sigma|^2\eta'+2\eps^4 a_{ij}h_i\partial_j|A_\Sigma|^2 \eta'
\end{aligned}
\end{equation}
\subsection{A gluing procedure}
We need to reduce the equation on $\R^{N+1}$ to a problem on the product space $\Sigma_\eps\times\R$. We introduce some useful notation. For a function $u:\Sigma_\eps\times\R\to\R$, we set
\begin{equation}\notag
u^{\natural}(\xi):=
\begin{cases}
u\circ X_{\eps,h}^{-1}(\xi),\qquad &\text{for $\xi\in\mathcal{N}_{\eps,h}\subset\R^{N+1}$}\\
\hspace{1cm}0, \qquad &\text{otherwise}.
\end{cases}
\end{equation}
On the other hand, given a function $v:\R^{N+1}\to\R$, we define for $(\y,t)\in\Sigma_\eps\times\R$,
\begin{equation}\notag
v^{\sharp}(\y,t):=\begin{cases}
v\circ X_{\eps,h}(\y,t)\qquad&\text{if $(\y,t)\in X^{-1}_{\eps,h}(\mathcal{N}_{\eps,h})$}\\
\hspace{1cm}0,\qquad &\text{otherwise.}
\end{cases}
\end{equation}
The notation just introduced can be explained as follows: $u^\natural$ refers to an expression of $u$ in {\it{natural}} or Euclidean coordinates while $u^\sharp$ refers to a function expressed in shifted Fermi coordinates.\\ 

Given a smooth cutoff function $\chi:\R\to[0,1]$ such that $\chi=0$ in $(-\infty,1)$ and $\chi=1$ in $(2,\infty)$ we set, for $i\ge 1$, $$\chi_i(\y,t):=\chi\left(|t|-\frac{\delta}{3\eps}-\frac{\delta\tilde{\eta}}{\eps}\log(|\eps\y|+2)-i\right),$$
with $\tilde{\eta}>0$ to be determined later (see Lemma \ref{lemma-error-far}) and we define a global approximation of the form
\begin{equation}\notag
w(x):=\chi^\natural_4 U^\natural+(1-\chi^\natural_4)\mathbb{H},
\end{equation}
where $\mathbb{H}=1$ in the connected component of $\R^{N+1}\backslash\Sigma_\eps$ containing the origin, $\mathbb{H}=-1$ in the other connected component of $\Sigma_\eps$. Then we look for a solution of the form $$u=w+\varphi,\qquad\varphi=\chi^\natural_2 \phi^\natural+\psi$$
where the unknowns are the functions $\phi:\Sigma_\eps\times\R\to\R$ and $\psi:\R^{N+1}\to\R$. Using that $\chi^\natural_1\chi^\natural_2=\chi^\natural_1$, we can rewrite the Allen-Cahn equation in the form
\begin{equation}\notag
\begin{aligned}
0&=S(w)-\Delta\varphi+W''(w)\varphi+Q_w(\varphi)=\\
&=\chi^\natural_2(-\Delta\phi^\natural+W''(U^\natural)\phi^\natural+\chi^\natural_1(W''(U^\natural)-W''(a))\psi+\chi^\natural_1 Q_{U^\natural}(\phi^\natural+\psi)+S(U^\natural))\\
&-\Delta\psi+((1-\chi^\natural_1)W''(w)+\chi^\natural_1 W''(a))\psi+(1-\chi^\natural_2)S(w)\\
&+(1-\chi^\natural_1)Q_w(\chi^\natural_2 \phi^\natural+\psi)-2\nabla \phi^\natural\cdotp
\nabla\chi^\natural_2-\Delta\chi^\natural_2\phi^\natural,
\end{aligned}
\end{equation}
where we have set $S(w):=-\Delta w+W'(w)$ and $Q_w(t):=3wt^2+t^3$. As a consequence, the Allen-Cahn equation (\ref{allen-cahn}) is solved if we find a solution $(h,\phi,\psi)$ of the coupled system
\begin{equation}
\label{eq-far}
\begin{aligned}
&-\Delta\psi+((1-\chi^\natural_1)W''(w)+\chi^\natural_1 W''(a))\psi+(1-\chi^\natural_2)S(w)\\
&+(1-\chi^\natural_1)Q_w(\chi^\natural_2 \phi^\natural+\psi)-2\nabla \phi^\natural\cdotp
\nabla\chi^\natural_2-\Delta\chi^\natural_2\phi^\natural=0\qquad\text{in $\R^{N+1}$}\\
\end{aligned}
\end{equation}
\begin{equation}
\label{eq-near}
-\Delta_{\Sigma_\eps}\phi-\partial^2_t\phi+W''(v_\star)\phi+B(\phi)+\chi_1(W''(U)-W''(a))\psi^\sharp+\chi_1 Q_{U}(\phi+\psi^\sharp)+S(U)=0
\end{equation}
where
$$B(\phi):=\chi_3(-\Delta+\Delta_{\Sigma_\eps}+\partial^2_t+W''(U)-W''(v_\star))\phi,$$
where the Laplacian is expressed in Fermi coordinates.\\

The strategy of the proof is the following. First we fix $h$ and $\phi$ and we look for a solution $\psi:=\psi(\phi,h)$ to equation (\ref{eq-far}), using the coercivity of the operator $-\Delta+(1-\chi^\natural_1)W''(w)-\chi^\natural_1 W''(1)$, then we plug the solution into equation (\ref{eq-near}) and we solve it in $(\phi,h)$. Equation (\ref{eq-near}) can be rewritten as
\begin{equation}
\label{eq-LS}
-\Delta_{\Sigma_\eps}\phi-\partial^2_t\phi+W''(v_\star)\phi=\mathcal{F}_\eps(\y,t,\phi,h)\qquad\text{in $\Sigma_\eps\times\R$,}
\end{equation}
where we have set $$\mathcal{F}_\eps(\y,t,\phi,h):=-(B(\phi)+\chi_1(W''(U)-W''(a))\psi^\sharp+\chi_1 Q_{U}(\phi+\psi^\sharp)+S(U)).$$
Since the operator $-\Delta_{\Sigma_\eps}-\partial^2_t+W''(v_\star)$ has a one-dimensional Kernel generated by $v'_\star$, it is possible to solve equation (\ref{eq-LS}) up to a Lagrange multiplier. More precisely, for fixed $h$, it is possible to find a solution $\phi=\phi(h)$ to the problem
\begin{equation}
\label{eq-aux}
\begin{aligned}
-\Delta_{\Sigma_\eps}\phi-\partial^2_t\phi+W''(v_\star)\phi&=\mathcal{F}_\eps(\y,t,\phi,h)+\lambda(\y)v'_\star(t)\\
\int_\R \phi(\y,t)v'_\star(t)dt&=0\qquad\forall\, \y\in\Sigma_\eps,
\end{aligned}
\end{equation}
where 
\begin{equation}
\label{def-lambda}
\lambda(\y):=\frac{1}{\int_\R(v'_\star)^2}\int_\R \mathcal{F}_\eps(\y,t,\phi,h) v'_\star(t)dt.
\end{equation}
depends on $h$. To conclude the proof, we need to find $h$ such that 
\begin{equation}
\label{eq-bifo}
\lambda(\y)=0.
\end{equation}
Equation (\ref{eq-aux}) is known as the \textit{auxiliary equation} and equation (\ref{eq-bifo}) is known as the \textit{bifurcation equation}. this technique is called the \textit{Lyapunov-Schmidt reduction}. All these steps will be explained in the next subsection.

\subsection{The Lyapunov-Schmidt reduction}

For $\beta\in(0,1)$, $\mu\in\R$ and functions $g\in C^{0,\beta}_{loc}(\R^{N+1})$, we introduce the norm
\begin{equation}\notag
\|g\|_{\infty, \mu}:=\|(|\eps x|^2+2)^{-\frac{\mu}{2}}g\|_{L^\infty(\R^{N+1})}.
\end{equation}
Moreover, we say that $g\in{\tt Y}^{{\natural,\beta}}_{{\mu}}$ if it is $S$-invariant and the norm
\begin{equation}\notag
\|g\|_{{\tt Y}^{{\natural,\beta}}_{{\mu}}}:=\sup_{\xi\in\R^{N+1}}(2+|\eps x|^2)^{-\frac{\mu}{2}}\|g\|_{C^{0,\beta}(B_1(x))}
\end{equation}
is finite. We also say that a function $\psi\in C^{2,\beta}_{loc}(\R^{N+1})$ is in ${\tt X}^{{\natural,\beta}}_{{\mu}}$ if it is $S$-invariant and the norm
\begin{equation}
\label{norm_C2pmu}
\|\psi\|_{{\tt X}^{{\natural,\beta}}_{{\mu}}}:=\sup_{x\in\R^{N+1}}(2+|\eps x|^2)^{-\frac{\mu}{2}}\|D^2\psi\|_{C^{0,\beta}(B_1(x))}
+\|\nabla\psi\|_{{\infty,\mu}}+\|\psi\|_{{\infty,\mu}}
\end{equation}
is finite.
Given $\beta\in (0,1)$, $\gamma,\,\mu\in\R$ and a function $f\in C^{0,\beta}_{loc}(\Sigma_\eps\times\R)$, we define the norm
\begin{equation}\notag
\|f\|_{\infty, \mu,\gamma}:=(|\eps\y|^2+2)^{-\frac{\mu}{2}}\cosh(t)^{-\gamma} f\|_{L^\infty(\Sigma_\eps\times\R)},
\end{equation}
Furthermore, we say that $f\in Y^{\sharp,\beta}_{\mu,\gamma}$ if it is $S$-invariant and the norm
\begin{equation}\notag
\|f\|_{Y^{{\sharp,\beta}}_{{\mu, \gamma}}}:=\sup_{\y\in \Sigma_\eps,\,t\in\R}(|\eps\y|^2+2)^{-\frac{\mu}{2}}\cosh(t)^{-\gamma} \|f\|_{C^{0,\beta}(I_{\y,t})},\qquad I_{\y,t}:=B_1(\y)\times(t,t+1) 
\end{equation}
is finite. Besides, for $S$-invariant functions $\phi\in C^{2,\beta}_{loc}(\Sigma_\eps\times\R)$, we say that $\phi\in X^{{\sharp,\beta}}_{{\mu, \rho}}$ if the norm
\begin{equation}\notag
\|\phi\|_{X^{{\sharp,\beta}}_{{\mu, \gamma}}}:=\sup_{\y\in\Sigma_\eps,\,t\in\R}(|\eps\y|^2+2)^{-\frac{\mu}{2}}\cosh(t)^{-\gamma} \|D^2\phi\|_{C^{0,\beta}(I_{\y,t})}+\|\nabla\phi\|_{{\infty,\mu,\gamma}}
+\|\phi\|_{{\infty,\mu,\gamma}}
\end{equation}
is finite.
\begin{lemma}
\label{lemma-error-far}
In the above notations, the error fulfils
$$|(1-\chi^\natural_2)S(w)|\le C e^{-\frac{\sqrt{2}\delta}{3\eps}}(|\eps x|^2+2)^{-\frac{\tilde{\eta}\delta\sqrt{2}}{2}}.$$
\end{lemma}
\begin{remark}
As a consequence of Lemma \ref{lemma-error-far}, we need to take $\tilde{\eta}\ge -\frac{\bar{\mu}\delta}{\sqrt{2}}$, with $\bar{\mu}:=\min\{-3,2\nu\}<0$\color{black}, in order to get $$|(1-\chi^\natural_2)S(w)|\le C e^{-\frac{\sqrt{2}\delta}{3\eps}}(|\eps x|^2+2)^{\bar{\nu}/2}.$$
\end{remark}
\begin{proof}
Using that $1-\chi^\natural_2$ is supported in the region where $$\frac{\delta}{3\eps}+\tilde{\eta}\log(|\eps \y|+2)+3<|t|<\frac{\delta}{3\eps}+\tilde{\eta}\log(|\eps \y|+2)+4$$
we have
\begin{equation}\notag
|(1-\chi^\natural_2)S(w)|\le Ce^{-\sqrt{2}\left(\frac{\delta}{3\eps}+\tilde{\eta}\log(|\eps \y|+2)\right)}\le Ce^{-\frac{\sqrt{2}\delta}{3\eps}}(|\eps \y|+2)^{-\sqrt{2}\tilde{\eta}}
\end{equation}
To conclude the proof, we note that
\begin{equation}\notag
\begin{aligned}
|\eps \y|&=|\eps x-\eps(t+h(\eps \y))\nu_\Sigma(\eps \y)|\ge |\eps x|-\eps|t+h(\eps \y)|\\
&\ge|\eps x|-\eps\delta\left(\frac{1}{2\eps}+\frac{1}{3}\right)-\eta\log(|\eps \y|+2)-4\ge c_1|\eps x|-c_2|\eps \y|,
\end{aligned}
\end{equation}
for some constants $c_i>0$.
\end{proof}
\begin{proposition}
Let $\alpha\in(0,\frac{1}{2}),\,\gamma\in(0,\sqrt{W''(a)}),\,C_1,\,C_2>0$ be fixed constants. Let $h\in C^{2,\alpha}_\nu(\Sigma)$ be such that
$$\|h\|_{\mathcal{C}^{2,\alpha}_{\nu}(\Sigma)}<C_1\eps.$$
Let $\phi\in X^{\sharp,\alpha}_{\bar{\mu},\gamma}$ be such that 
$$\|\phi\|_{X^{\sharp,\alpha}_{\bar{\mu},\gamma}}<C_2\eps^3,$$ 
Then there exists a unique solution $\psi=\psi(\phi,\bar{h})\in X^{\natural,\alpha}_{\bar{\mu}}$ to equation (\ref{eq-far}) satisfying
\begin{equation}\notag
\begin{aligned}
\|\psi(\phi,h)\|_{{\tt X}^{\natural,\alpha}_{\bar{\mu}}}  &\le C_3 e^{-\frac{\beta}{\eps}},\\
\|\psi(\phi^1,h)-\psi(\phi^2,h)\|_{{\tt X}^{\natural,\alpha}_{\bar{\mu}}}&\le ce^{-\frac{\beta}{\eps}}\|\phi^1-\phi^2\|_{{\tt X}^{\sharp,\alpha}_{\bar{\mu}, \rho}}\\
\|\psi(\phi,h^1)-\psi(\phi,h^2)\|_{{\tt X}^{\natural,\alpha}_{\bar{\mu}}}&
\le ce^{-\frac{\beta}{\eps}}\|h^1-h^2\|_{C^{2,\alpha}_{\nu}(\Sigma)}
\end{aligned}
\end{equation}
for some $c,\,C_3,\,\beta>0$. 
\label{prop_psi}
\end{proposition}
The proof of Proposition \ref{prop_psi} basically consists of two steps. First we consider the linear problem
\begin{equation}
\label{eq-far-lin}
-\Delta \psi+V_\eps(x)\psi=g\qquad\text{in $\R^{N+1}$,}
\end{equation}
where $V_\eps(x):=(1-\chi^\natural_1)W''(w)-\chi^\natural_1 W''(a)$ and we find a solution $\psi={\tt F}_3(g)\in {\tt X}^{\natural,\alpha}_\mu$, for any $g\in {\tt Y}^{\natural,\alpha}_\mu$ and $\mu<0$. In other words, we construct a right inverse of the operator $-\Delta+V_\eps$ in the suitable spaces. Then we use this right inverse to solve (\ref{eq-far}) by a fixed point argument.
\begin{lemma}
\label{lemma-eq-far-lin}
Let $\mu<0,\,\alpha\in(0,\frac{1}{2})$ and $g\in {\tt Y}^{\natural,\alpha}_\mu$. Then there exists a solution $\psi:={\tt F}_1(g)\in {\tt X}^{\natural,\alpha}_\mu$ to (\ref{eq-far-lin}) such that
$$\|\psi\|_{{\tt X}^{\natural,\alpha}_\mu}\le c\|g\|_{{\tt Y}^{\natural,\alpha}_\mu}.$$
\end{lemma}
The proof of Lemma \ref{lemma-eq-far-lin} follows step by step the proof of Proposition $5.2$ of \cite{agudelo2022doubling}. Finally, the proof of Proposition \ref{prop_psi} follows from a fixed point argument.
\begin{proof}
Equation (\ref{eq-far}) can be formulated as a fixed point problem of the form
$$\psi=-{\tt F}_1\left((1-\chi^\natural_2)S(w)+(1-\chi^\natural_1)Q_w(\chi^\natural_2 \phi^\natural+\psi)-2\nabla \phi^\natural\cdotp
\nabla\chi^\natural_2-\Delta\chi^\natural_2\phi^\natural\right).$$
Using Lemma \ref{lemma-error-far} and the contraction mapping Theorem, it is possible to see that there exists a unique solution in the ball
$$B_{C_3}:=\{\psi\in {\tt X}^{\natural,\alpha}_{\bar{\nu}}:\,\|\psi\|\le C_3 e^{-\frac{\sqrt{2}\delta}{3\eps}}\},$$
provided $C_3>0$ is large enough.
\end{proof}
Now we solve the auxiliary equation (\ref{eq-aux}). First we introduce the spaces
\begin{equation}\notag
\begin{aligned}
X^{\sharp,\alpha}_{\bot,\mu,\gamma}&:=\left\{\phi\in X^{\sharp,\alpha}_{\mu,\gamma}:\,\int_\R \phi(\y,t)v'_\star(t)dt=0,\,\forall\y\in\Sigma_\eps\right\},\\
Y^{\sharp,\alpha}_{\bot,\mu,\gamma}&:=\left\{\phi\in Y^{\sharp,\alpha}_{\mu,\gamma}:\,\int_\R \phi(\y,t)v'_\star(t)dt=0,\,\forall\y\in\Sigma_\eps\right\},
\end{aligned}
\end{equation}
for $\gamma\in(-\sqrt{W''(a)},0)$ and $\mu\in\R$.
\begin{proposition}
\label{prop-phi}
Let $\alpha\in(0,\frac{1}{2}),\,\gamma\in (-\sqrt{W''(a)},0),\,C_2>0$ be fixed constants. Let $h\in C^{2,\alpha}_\nu(\Sigma)$ be such that
$$\|h\|_{\mathcal{C}^{2,\alpha}_{\nu}(\Sigma)}<C_1\eps.$$
Then there exists a unique solution $\phi=\phi(h)\in X^{\sharp,\alpha}_{\bot,\bar{\mu},\gamma}$ to the auxiliary equation (\ref{eq-aux}) such that
\begin{equation}\notag
\begin{aligned}
\|\phi(h)\|_{X^{\sharp,\alpha}_{\bar{\mu},\gamma}}&\le \Lambda_1\eps^3\\
\|\phi(h^1)-\phi(h^2)\|_{X^{\sharp,\alpha}_{\bar{\mu},\gamma}}&\le c\eps^3\|h^1-h^2\|_{C^{2,\alpha}_{\nu}(\Sigma)},
\end{aligned}
\end{equation}
for some $c,\,C_2>0$.
\end{proposition}
In order to prove Proposition \ref{prop-phi}, first we need to consider the linear problem
\begin{equation}
\label{eq-aux-lin}
\begin{aligned}
-\Delta_{\Sigma_\eps}\phi-\partial^2_t \phi+W''(v_\star(t))\phi&=g \qquad\text{in $\Sigma_\eps\times\R$}\\
\int_\R \phi(\y,t)v'_\star(t)dt&=0\qquad\forall\,\y\in\Sigma_\eps
\end{aligned}
\end{equation}
for right-hand sides $g\in Y^{\sharp,\alpha}_{\bot,\bar{\mu},\gamma}$. 
\begin{lemma}
\label{lemma-eq-aux-lin}
Let $\beta\in(0,1)$, $\gamma\in(-\sqrt{W''(a)},0)$ and $\mu<0$. Then, for any $f\in Y_{\bot,\mu,\gamma}^{\sharp,\beta}$, there exists a unique solution $\phi:={\tt F}_2(f)\in X_{\bot,\mu,\gamma}^{\sharp,\beta}$ to (\ref{eq-aux-lin}). Moreover
$$\|\phi\|_{X^{\sharp,\beta}_{\mu,\gamma}}\le c\|f\|_{Y^{\sharp,\beta}_{\mu,\gamma}}.$$
\end{lemma}
This is the analogue of Proposition $5.4$ of \cite{agudelo2022doubling}, the proof is identical. Using Lemma \ref{lemma-eq-aux-lin} and a fixed point argument, it is possible to prove Proposition \ref{prop-phi}.
\begin{proof}
Using the right inverse ${\tt F}_2$ constructed in Lemma \ref{lemma-eq-aux-lin}, it is possible to formulate equation (\ref{eq-aux}) as a fixed point problem in the form
$$\phi=-{\tt F}_2\left(B(\phi)+\chi_1(W''(U)-W''(a))\psi^\sharp+\chi_1 Q_{U}(\phi+\psi^\sharp)+S(U)\right).$$
Due to the expansions (\ref{est-error}), (\ref{def-E1}), (\ref{def-E2}) and the contraction mapping Theorem, equation (\ref{eq-aux}) can be uniquely solved in the ball
$$B_{C_2}:=\{\phi\in X^{\sharp,\alpha}_{\bot,\bar{\mu},\rho}:\,\|\phi\|_{X^{\sharp,\alpha}_{\bar{\mu},\gamma}}<C_2\eps^3\}$$
provided $C_2>0$ is large enough.
\end{proof}
In order to conclude the construction of the required solution $u_\eps$ to equation \eqref{allen-cahn}, we need to the bifurcation equation \eqref{eq-bifo} with respect to $h$. Computing $\lambda(\y)$ explicitly (see \ref{eq-bifo}), using (\ref{def-E1}), (\ref{def-E2}) and the properties of $\phi$ and $\psi$ given by Propositions \ref{prop-phi} and \ref{prop_psi}, it is possible to rewrite equation (\ref{eq-bifo}) in the form 
\begin{equation}
\label{eq-bifo-proj}
J_\Sigma h=\mathcal{G}_\eps(y,h)\qquad\text{in }\Sigma.   
\end{equation}
We note that $\mathcal{G}_\eps(\cdotp,h)\in C^{0,\alpha}_{\nu-2}(\Sigma)$ for any $\nu\ge-1$. This estimate comes from the computations. In fact, the largest term in $\eps$ is given by $\eps{\rm tr}(A_\Sigma^3)=O(|y|^{-3})$ as $|y|\to\infty$, while the largest term in $|y|$ is given $|\nabla_\Sigma h|^2=O(|y|^{2(\nu-1)})$ as $|y|\to\infty$.
\begin{proposition}
\label{prop-bifo}
Assume that \eqref{choice-nu} is satisfied
and $\alpha\in (0,\frac{1}{2})$. Then there exists a unique solution $h\in C^{2,\alpha}_{\nu}(\Sigma)$ to the bifurcation equation (\ref{eq-bifo}) such that $\|h\|_{C^{2,\alpha}_{\nu}(\Sigma)}<C_1\eps$, for some $C_1>0$.
\end{proposition}
\begin{proof}
Using that $\mathcal{G}_\eps(\cdotp,h)\in C^{0,\alpha}_{\nu-2}(\Sigma)$ (see equation \eqref{eq-bifo-proj} for the definition of $\mathcal{G}_\eps(\cdotp,h)$), since $\nu\ge-1$, point $(1)$ of Proposition $3.2$ of \cite{agudelo2025fredholm} and the restrictions about $\nu$, equation (\ref{eq-bifo-proj}) can be written as 
$$h=\Phi(\mathcal{G}_\eps(\cdotp,h)),$$
where $\Phi$ is the right inverse of $J_\Sigma$ constructed in Theorem \ref{prop-right-inverse-Jacobi}. The latter is uniquely solvable in the ball
$$B_{C_1}:=\{h\in C^{2,\alpha}_{\nu}(\Sigma):\,\|h\|_{C^{2,\alpha}_{\nu}(\Sigma)}<C_1\eps\},$$
provided $C_1>0$ is large enough. This can be done applying a fixed point argument.
\end{proof}
\begin{remark}
Note that the solution $h$ found in Proposition \ref{prop-bifo} decays as $|y|\to\infty$, since $\nu<0$. As a consequence, property $(1)$ of Theorem \ref{main-th} is fulfilled.
\end{remark}

\section{Morse index}

Now we deal with the Morse index of the solutions constructed in Theorem \ref{main-th}. In this part we use the conventions and notations introduced in section 5.3. First, we recall the quadratic forms in \eqref{def:quadr_Sigma} and \eqref{def:quadr_sln_u}, namely
$$\mathcal{Q}_\Sigma(\varphi):=\int_\Sigma (|\nabla_\Sigma\varphi|^2-|A_\Sigma|^2\varphi^2)d\sigma$$
defined for $\varphi\in C^\infty_c(\Sigma)$ and
$$\mathbb{Q}_{u_\eps}(\phi):=\int_{\R^{N+1}}(|\nabla \phi|^2+W''(u_\eps)\phi^2) dx$$
defined for $\phi\in C^\infty_c(\R^{N+1})$. Recall also that 
$$Morse(\Sigma):=\sup\left\{\dim(X):\,X\subset C^\infty_c(\Sigma),\mathcal{Q}_{\Sigma}(\varphi)<0,\,\,\forall\,\varphi\in X\setminus\{0\}\right\}$$
and 
$$Morse(u_\eps):=\sup\left\{\dim(Y):\,Y\subset C^\infty_c(\R^{N+1}),\,\Q_{u_\eps}(\phi)<0,\,\,\forall\,\phi\in Y\setminus\{0\}\right\}.$$

We need to prove the inequalities
\begin{equation}\label{ineq-Morse1}
Morse(u_\eps)\ge Morse (\Sigma)
\end{equation}
and
\begin{equation}\label{ineq-Morse2}
Morse(u_\eps)\le Morse (\Sigma).
\end{equation}
We stress that, in order for \eqref{ineq-Morse2} to be true, we need the further assumptions that $\Sigma$ is dilation-nondegenerate and $\bar{\nu}\ge -\frac{N-2}{2}$.\\

Let us first introduce some notation.\\

For $y\in\Sigma$, we set $p_0(y):=(1+|y|)^{-2}$. Note that $p_0>0$ and there exists $c>0$ such that 
$$0\le |A_\Sigma(y)|^2\le cp_0(y),\qquad\forall\,y\in\Sigma.$$ 
Moreover, we introduce the space $$\mathcal{H}(\Sigma):=\{\varphi\in H^1_{loc}(\Sigma):|\nabla_\Sigma\varphi|,\,p_0^{1/2}\varphi\in L^2(\Sigma)\}.$$
which is a Hilbert space with respect to the scalar product
$$\langle \varphi,\psi\rangle_{\mathcal{H}(\Sigma)}:=\int_{\Sigma}\nabla\varphi\cdotp\nabla\psi\,d\sigma+\int_{\R^{N+1}}p_0(y)\varphi\psi\,d\sigma,\qquad\forall\,\varphi,\psi\in\mathcal{H}(\Sigma).$$
Moreover, we introduce the weight function $$p(x):=\chi_2^\natural (p_0(\eps\cdotp))^\natural+(1-\chi_2^\natural)(1+|\eps x|)^{-2}\quad\forall\,x\in\R^{N+1},\quad p_0(y):=(1+|y|)^{-2}\quad\forall\, y\in\Sigma,\,\eps>0.$$

\subsection{Proof of inequality (\ref{ineq-Morse1})}

\begin{lemma}\label{lemma:quasi_autofunzione_Allen-Cahn}
Let $\eps>0$ be small enough and let $\varphi\in C^2_c(\Sigma)$. Set $\phi(\y,t):=\varphi(\eps\y)v'_\star(t)$ for $(\y,t)\in \Sigma_\eps\times \R$. Let $v:=(\phi \chi_3)^\natural$ and $c_\star:=\int_\R (v'_\star)^2$. Then
\begin{enumerate}
\item 
\begin{equation}
\int_\R L(\phi)v'_\star\,dt=-\eps^2c_\star J_\Sigma\varphi(\eps\y)+f_\eps(\eps\y)\varphi(\eps \y)-\eps^2c_\star(a^{ij}\varphi_{ij}+b^i\varphi_i)(\eps\y)\qquad\forall\,\y\in\Sigma_\eps,
\label{pr-Sigma}
\end{equation}
where $f_\eps\in C^{0,\alpha}(\Sigma)$ is such that
\begin{equation}\label{est-f-eps}
|f_\eps(y)|\le c\eps^3(1+|y|)^{2(\nu-1)}\qquad\forall\, y\in\Sigma.    
\end{equation}
\item
\begin{equation}
\label{exp-Qu}
\mathbb{Q}_{u_\eps}(v)=\eps^{2-N} c_\star \mathcal{Q}_\Sigma(\varphi)+o_\eps(\eps^{2-N})\|\varphi\|^2_{\mathcal{H}(\Sigma)}.
\end{equation}
\end{enumerate}
    
\end{lemma}

\begin{proof}
\begin{enumerate}
\item Using the expansion of the Laplacian in shifted Fermi coordinates (\ref{Laplacian_shifted_Fermi}), we compute 
\begin{equation}\label{exp-L-psi-parallel}
\begin{aligned}
\Delta_{\y,t}(\varphi v'_\star)-W''(v_\eps)\varphi v'_\star&=\eps^2\Delta_\Sigma \varphi v'_\star+\varphi v'''_\star-\eps^2  \varphi J_\Sigma h v''_\star-\eps^2|A_\Sigma|^2\varphi tv''_\star-2\eps^2\nabla_\Sigma h\cdotp\nabla_\Sigma\varphi v''_\star\\
&+\eps^2|\nabla_\Sigma h|^2 \varphi v'''_\star-\eps\Q\varphi v''_\star+\eps^2(a^{ik}\varphi_{ik}+b^i\varphi_i)v'_\star-\eps^2(a^{ik}h_{ik}+b^i h_i)\varphi v''_\star\\
&-2\eps^2 a^{ik}\varphi_i h_k v''_\star+\eps^2 a^{ik}h_i h_k \varphi v'''_\star-W''(v_\star)\varphi v'_\star-\eps^2|A_\Sigma|^2W'''(v_\star)\eta\varphi v'_\star\\
&-(W''(v_\eps)-W''(v_\star)-\eps^2|A_\Sigma|^2 W'''(v_\star)\eta)\varphi v'_\star,
\end{aligned}
\end{equation}
where $\Delta_{\y,t}$ represents the expression of the Laplacian in shifted Fermi coordinates.\\

Moreover, differentiating the equation satisfied by $\eta$, we can see that
$$-\eta'''+W''(v_\star)\eta'+W'''(v_\star)v'_\star\eta=-tv''_\star-v'_\star,$$
so that multiplying by $v'_\star$ and integrating over $\R$ we have
$$\int_\R W'''(v_\star)\eta (v'_\star)^2 dt=-\int_\R tv''_\star v'_\star dt-\int_\R (v'_\star)^2 dt=-\frac{1}{2}\int_\R (v'_\star)^2 dt.$$
In addition, using the form of the remainder in the construction of the solution to equation \eqref{allen-cahn}, we have
\begin{equation}\notag
\int_\R (W''(v_\eps)-W''(v_\star)-\eps^2|A_\Sigma|^2 W'''(v_\star)\eta)\varphi_R(\eps\y) v'_\star\,dt=: \tilde{f}_\eps(\eps\y)\varphi_R(\eps\y),
\end{equation}
where
\begin{equation}\notag
|\tilde{f}_\eps(y)|\le c\eps^3(1+|y|)^{-3},\qquad\forall\,y\in\Sigma,
\end{equation}
for some constant $c>0$ independent of $R$.\\
As a consequence, multiplying the expansion of $L(\varphi(\eps\cdotp)v'_\star)$ by $v'_\star$, integrating by parts and using the decay of the geometric quantities along $\Sigma$, we conclude the proof of \eqref{pr-Sigma}.

\item Multiplying \eqref{pr-Sigma} by $\varphi(\eps\y)$ and multiplying over $\Sigma$, we can see that
\begin{equation}\notag
\begin{aligned}
\Q_{u_\eps}(\phi)&=-\eps^2 c_\star \int_{\Sigma_\eps} J_\Sigma \varphi(\eps\y)\varphi(\eps\y) d\sigma(\y)+\int_{\Sigma_\eps}f_\eps(\eps\y)\varphi(\eps\y)^2d\sigma\\
&-\eps^2\int_{\Sigma_\eps} (a^{ij}\varphi_{ij}+b^i\varphi_i)\varphi(\eps\y)\varphi(\eps\y)d\sigma+o_\eps(\eps^{2-N})\|\varphi\|_{\mathcal{H}(\Sigma)}\\
&=\eps^{2-N} c_\star \mathcal{Q}_\Sigma(\varphi)+o_\eps(\eps^{2-N})\|\varphi\|^2_{\mathcal{H}(\Sigma)}
\end{aligned}
\end{equation}
The integrations by parts over $\Sigma$ are justified by the fact that $\varphi$ is compactly supported. The terms involving $a^{ij}$ and $b^i$ fulfill
$$\int_{\Sigma} (a^{ij}\varphi_{ij}+b^i\varphi_i)(y)\varphi(y)d\sigma\le c\eps\|\nabla_\Sigma\varphi\|^2_{L^2(\Sigma)},$$
thanks to the decays of such coefficients in $\eps$ and $\y$ (see \eqref{est_remainder_laplacian}). The term involving $f_\eps$ is estimated by using \eqref{est-f-eps}. A similar argument is used in \cite{del2013entire} and \cite{agudelo2015solutions}, 
\end{enumerate}
\end{proof}

\begin{proposition}
${\rm Morse}(u_\eps)\ge {\rm Morse}(\Sigma)$.
\end{proposition}
\begin{proof} Let $m\in \mathbb{N}$ and $X$ be a $m$-dimensional subspace of $C^\infty_c(\Sigma)$ such that 
$$\mathcal{Q}_\Sigma(\varphi)<0\qquad\forall\,\varphi\in X\setminus\{0\}.
$$

Set
$$
\underline{\lambda}:=\inf \limits_{\varphi\in X\setminus \{0\}}\frac{\mathcal{Q}_{\Sigma}(\varphi)}{\int_{\Sigma}p_0\varphi^2d\sigma} \hbox{ and } \overline{\lambda}:=\sup\limits_{\varphi\in X\setminus \{0\}}\frac{\mathcal{Q}_{\Sigma}(\varphi)}{\int_{\Sigma}p_0\varphi^2d\sigma}. 
$$

Being $X$ finite-dimensional, a homogeneity and a continuity argument yields $-\infty<\underline{\lambda}\leq \overline{\lambda}<0$.

\medskip
Let $\varphi\in X\setminus\{0\}$. Multiplying if necessary, by a constant, we can assume that $\|p_0^{1/2}\varphi\|_{L^2(\Sigma)}=1$. Setting ${\tt v}(\y,t):=\varphi(\eps\y)v'_\star(t)$ and  $v:=({\tt v} \chi_3)^\natural$ by Lemma \ref{lemma:quasi_autofunzione_Allen-Cahn}, we can see that $v\in C^\infty_c(\R^{N+1})\setminus\{0\}$ and it satisfies
\begin{equation}\notag
\Q_{u_\eps}(\phi)
=\eps^{2-N} c_\star \mathcal{Q}_\Sigma(\varphi)+o_\eps(\eps^{2-N})\|\varphi\|^2_{\mathcal{H}(\Sigma)}<\overline{\lambda}\eps^{2-N}\int_\Sigma p_0\varphi^2 d\sigma(y)+o_\eps(\eps^{2-N})\le\overline{\lambda}\eps^{2-N}(1+o_\eps(1))<0
\end{equation}
provided $\eps>0$ is small enough. As a consequence, we get an $m$-dimensional subspace $Y\subset C^\infty_c(\R^{N+1})$ such that
$$\Q_{u_\eps}(\phi)<0\qquad\forall\,\phi\in Y,$$
which concludes the proof.
\end{proof}
\begin{remark}
Note that, in this subsection, we have never used the fact that $\bar{\nu}\ge-\frac{N-2}{2}$. For this reason this restriction is not necessary to conclude that $Morse(u_\eps)\ge Morse(\Sigma)$.
\end{remark}
\subsection{Proof of inequality (\ref{ineq-Morse2})}
In this subsection we will prove that, for $\eps>0$ small enough, we have $${\rm Morse}(u_\eps)\le {\rm Morse(\Sigma)},$$
at least if $\Sigma$ is dilation-nondegenerate and $\bar{\nu}\ge-\frac{N-2}{2}$.\\

In order to do so, we consider the eigenvalues of the truncated problem
\begin{equation}
\label{tr-prob-BR}
\begin{cases}
-\Delta \phi+W''(u_\eps)\phi=\lambda p(x)\phi \qquad\text{in}\, B_R(0) \\
\phi\in H^1_0(B_R(0)),
\end{cases}
\end{equation}
where $R>2\eps^{-1}$ is fixed. We note that the eigenvalues $\lambda_{i,\eps}(R)$ and the corresponding eigenfunctions $\phi_{i,\eps,R}$ have a variational characterization through the Rayleight quotient, 
thanks to the boundedness of $B_R(0)$. More precisely the first eigenvalue satisfies
\begin{equation}
\label{def-lambda-eps-1R}
\lambda_{\eps,1}(R):=\inf\left\{\frac{\Q_{u_\eps}(\phi)}{\|p^{1/2}\phi\|_{L^2(B_R(0))}}:\,\phi\in H^1_0(B_R(0))\setminus\{0\}\right\} . 
\end{equation}
and, for $i\ge 2$, the $i$-the eigenvalue satisfies
\begin{equation}
\label{def-lambda-eps-iR}
\lambda_{\eps,i}(R):=\inf\left\{\frac{\Q_{u_\eps}(\phi)}{\|p^{1/2}\phi\|_{L^2(B_R(0))}}:\,\phi\in H^1_0(B_R(0))\setminus\{0\},\,\int_{B_R(0)}p(x)\phi\phi_j dx=0,\,1\le j< i\right\}.
\end{equation}
Due to the linearity of \eqref{tr-prob-BR}, we can assume that 
\begin{equation}\label{norm-phi=1}
\int_{\R^{N+1}}p(x)\phi_{\eps,i,R}^2\,dx=1
\end{equation}
By the Calder\'{o}n-Zygmund regularity estimates (see Lemma $9.17$ of \cite{gilbarg1998elliptic}), the Sobolev embeddings and, if necessary, a bootstrap argument, it is possible to see that $\phi_{\eps,i,R}\in L^\infty(B_R)\cap C^{2}(B_R)$ is actually a classical solution to \eqref{tr-prob-BR}.\\

\color{black}Let $m_R(u_\eps)$ denote the number of negative eigenvalues (counting multiplicities) of \eqref{tr-prob-BR}. Then, due to the variational characterization of these eigenvalues, $m_R(u_\eps)$ is non-decreasing in $R$ and we have
\begin{lemma}
\label{lemma-sup-morse-uR}
\begin{equation}
\label{sup-morse-uR}
\sup_{R>2\eps^{-1}}m_R(u_\eps)=\lim_{R\to\infty}m_R(u_\eps)=Morse(u_\eps)=:m_\eps\in\N\cup\{\infty\}.
\end{equation}
\end{lemma}
\begin{proof}
Assume by contradiction that $k:=\lim_{R\to\infty}m_R(u_\eps)<m_\eps$. Due to the definition of Morse index, it is possible to find $k+1$ linearly independent functions $\bar{\phi}_{\eps,1},\dots,\bar{\phi}_{\eps,k+1}\in C^\infty_c(\R^{N+1})$ such that $\|p^{1/2}\bar{\phi}_{i,\eps}\|_{L^2(\R^{N+1})}=1$ and $\Q_{u_\eps}(\bar{\phi}_{\eps,i})<0$, for $i=1,\dots,k+1$. As a consequence, taking $R>0$ so large that ${\rm supp}(\bar{\phi}_{\eps,i})\subset B_R(0)$, for $i=1,\dots,k+1$, we have
$$\lambda_{\eps,i}(R)\le\Q_{u_\eps}(\bar{\phi}_{\eps,i})<0,\qquad\forall\,i=1,\dots,k+1,$$
a contradiction. Hence $k\ge m_\eps$. The opposite inequality follows at once from the definition of Morse index.
\end{proof}

The strategy of the proof of inequality \eqref{ineq-Morse2} is the following. If $Morse(\Sigma)=\infty$ there is nothing to prove, hence we can assume that $n:=Morse(\Sigma)$ is finite. Assume by contradiction that there exists a sequence $\eps_k\to 0$ such that $n<m_{\eps_k}$. As a consequence, there exists $R_k>2\eps_k^{-1}$ such that there exist $n+1$ linearly independent eigenfunctions $$\phi_{\eps_k,1,R_1},\dots,\phi_{\eps_k,n+1,R_k}\in H^1_0(B_{R_k}(0))$$
of \eqref{tr-prob-BR} in $B_{R_k}(0)$ fulfilling
\begin{equation}\label{ort-phi-R-ij}
\int_{\R^{N+1}}p(x)\phi_{\eps_k,i,R}\phi_{\eps_k,j,R}dx=\delta_{ij},\,1\le i,j\le n+1,\,k\in\N
\end{equation}
corresponding to negative eigenvalues
$$\lambda_{\eps_k,1}(R_k),\dots,\lambda_{\eps_k,n+1}(R_k)\in(-\infty,0).$$
To simplify the notation, from now on we will write $\phi_{i,k}:=\phi_{\eps_k,i,R_k},\,\lambda_{i,k}:=\lambda_{i,\eps_k}(R_k)$ and $\eps:=\eps_k$.\\

We will show that these eigenfunctions can be used to construct a Jacobi field $\varphi\ne 0$ of $\Sigma$ such that  $\varphi\in D_\nu(\Sigma)$ for some $\nu>\bar{\nu}$, which contradicts the dilation-nondegeneracy of $\Sigma$.\\



\color{black}Writing $\phi_{i,k}=\chi_2^\natural \phi_{i,k}+(1-\chi_2^\natural)\phi_{i,k}$, the compactly supported $C^2$-function $$\psi_{i,k}(\y,t):=\chi_2(\y,t)\phi_{i,k}^\sharp(\y,t),\qquad\forall\,(\y,t)\in\Sigma_\eps\times\R,$$ 
can be decomposed as
\begin{equation}\label{dec-psi}
\psi_{i,k}=\psi_{i,k}^\parallel+\psi_{i,k}^\bot,\qquad\int_\R \psi_{i,k}^\bot v'_\star dt=0\quad\forall\,\y\in\Sigma_\eps,\quad \psi_{i,k}^\parallel=\varphi_{i,k}(\eps\y)v'_\star(t),
\end{equation}
where the function $\varphi_{i,k}:\Sigma\to\R$ defined by $$\varphi_{i,k}(\eps \y):=\frac{\int_\R\psi_{i,k}(\y,t)v'_\star(t)dt}{\int_\R (v'_\star)^2dt}\qquad\forall\,\y\in\Sigma_\eps$$
is the projection of $\psi_k$ onto $v'_\star$.

\begin{lemma}\label{lemma-varphi-R-bd}
There exists $c>0$ such that
\begin{equation}\label{est-varphi-lu}
\eps^{-N/2}\|p_0^{1/2}\varphi_{i,k}\|_{L^2(\Sigma)}\le c    
\end{equation}
\begin{equation}\label{est-psi-R-bot}
\|p_0^{1/2}\psi_{i,k}^\bot\|_{L^2(\Sigma_\eps\times\R)}\le c, 
\end{equation}
and
\begin{equation}\label{est-phi-far}
\|(1-\chi^\natural_1)\phi_{i,k}\|_{H^1(B_R(0))}\le c\|\psi_{i,k}\|_{H^1({\rm supp}(\nabla_{\y,t}\chi_1))}
\end{equation}
for any $1\le i\le n+1$ and $k>0$ large enough.
\end{lemma}
\begin{proof}
Using the fact that $$\int_{\R^{N+1}}p\phi_{i,k}^2\, dx=1$$ 
and the change of variables $x=\y+(t+h(\eps\y))\nu_\Sigma(\eps\y)$ associated to the shifted Fermi coordinates in the above integral, we can see that
$$\int_{\Sigma_\eps\times\R}p_0(\eps\y)\psi_{i,k}^2(\y,t)d\y dt\le c,$$
for some constant $c>0$ independent of $k$. Using the decomposition $\psi_{i,k}=\psi_{i,k}^\parallel+\psi_{i,k}^\bot$ introduced in \eqref{dec-psi}, this yields that $\eps^{-N/2}p_0^{1/2}\varphi_{i,k}$ is bounded above in $L^2(\Sigma)$, uniformly in $k$, and  $p_0^{1/2}(\eps\cdotp)\psi_{i,k}^\bot$ is bounded above in $L^2(\Sigma_\eps\times\R)$ uniformly in $k$ too.\\

It remains to prove \eqref{est-phi-far}. In order to do so, we multiply the equation satisfied by $\phi_{i,k}$ by $\phi_{i,k}(1-\chi^\natural_1)^2$ and integrate by parts over $B_{R_k}(0)$ to see that
\begin{equation}\notag
\begin{aligned}
\int_{B_{R_k}(0)}(|\nabla \phi_{i,k}|^2+W''(u_\eps)\phi_{i,k}^2)(1-\chi^\natural_1)^2\,dx&=\lambda_{i,k}\int_{B_{R_k}(0)}p(x)\phi_{i,k}^2\, dx\\ 
&+2\int_{B_{R_k}(0)}(1-\chi^\natural_1)\phi_{i,k}\nabla\phi_{i,k}\cdotp\nabla \chi^\natural_1\,dx \\ 
&\le c\int_{B_{R_k}(0)\cap\{\nabla\chi^\natural_1\ne 0\}}|\phi_{i,k}||\nabla\phi_{i,k}|\,dx\\
&\le c\int_{{\rm supp}(\nabla_{\y,t}\chi_1)}|\psi_{i,k}||\nabla_{\y,t}\psi_{i,k}| \,d\sigma dt\le\\
&\le c\|\psi_{i,k}\|_{L^2({\rm supp}(\nabla_{\y,t}\chi_1))}\|\nabla_{\y,t}\psi_{i,k}\|_{L^2({\rm supp}(\nabla_{\y,t}\chi_1))}\\
&\le c\|\psi_{i,k}\|^2_{H^1({\rm supp}(\nabla_{\y,t}\chi_1))}
\end{aligned}    
\end{equation}
Therefore, the conclusion follows from the fact that $$W''(u_\eps)\ge\frac{1}{2}W''(a)>0\qquad\forall\,x\in{\rm supp}(\nabla \chi^\natural_{1}).$$
\end{proof}
\begin{remark}\label{rem-phi-far-bd}
The bound \eqref{est-phi-far} shows that the norm of $\phi_{i,k}$ far from $\Sigma_\eps$ can be controlled by the norm of $\psi_{i,k}$ in a stripe whose distance from $\Sigma_\eps$ is exponentially large in $\eps$. 
\end{remark}
Note that
$$-\Delta(\chi^\natural_2\phi_{i,k})+W''(u_\eps)\chi^\natural_2\phi_{i,k}=\lambda_{i,k}p(x)\chi^\natural_2\phi_{i,k}-2\nabla\chi^\natural_2\cdotp\nabla\phi_{i,k}-(\Delta\chi^\natural_2)\phi_{i,k}\quad\text{in}\,\R^{N+1}.$$
The operator $\tilde{L}_\eps$ is basically the restriction of $L$ to $\mathcal{X}$ composed with the projection onto $v'_\star$.\\

Let us introduce the spaces

\begin{equation}\notag
\begin{aligned}
\mathcal{H}(\R^{N+1})&:=\{\phi\in H^1_{loc}(\R^{N+1}):|\nabla\phi|,\,p^{1/2}\phi\in L^2(\R^{N+1})\},\\
p^{-1/2}L^2(\R^{N+1})&:=\{f\in L^2_{loc}(\R^{N+1}):\, p_0^{1/2}f\in L^2(\R^{N+1})\}
\end{aligned}
\end{equation}
and the first-order linear operator $${\tt L}_\eps:\mathcal{H}(\R^{N+1})\to p^{-1/2}L^2(\R^{N+1})$$ defined by the relation
$${\tt L}_\eps\phi=-2\nabla\chi^\natural_2\cdotp\nabla\phi-(\Delta\chi^\natural_2)\phi.$$
In these notations, the function $\psi_{i,k}$ fulfills
\begin{equation}\label{eq-psi-R}
L\psi_{i,k}=\lambda_{i,k}p_0(\eps\y) \psi_{i,k}+({\tt L}_\eps\phi_{i,k})^\sharp\qquad\text{in }\,\Sigma_\eps\times\R,  
\end{equation}
where $L:=-\Delta_{\y,t}+W''(v_\eps)$ is defined through the expansion of the Laplacian (see \eqref{Laplacian_shifted_Fermi}). 
Equivalently, due to the linearity of $L$, using the decomposition $$\psi_{i,k}=\psi_{i,k}^\parallel+\psi_{i,k}^\bot=\varphi_{i,k}(\eps\cdotp)v'_\star+\psi_{i,k}^\bot,$$ we have
\begin{equation}\label{eq-parallel-bot}
L(\varphi_{i,k}(\eps\cdotp)v'_\star+\psi_{i,k}^\bot)=\lambda_{i,k}p_0(\eps\cdotp) \psi_{i,k}+({\tt L}_\eps\phi_{i,k})^\sharp\qquad\text{in }\,\Sigma_\eps\times\R,    
\end{equation}
The latter relation can be used to derive the equation satisfied by $\varphi_R$.\\

For this purpose, let us introduce the spaces 
\begin{equation}\notag
\begin{aligned}
\mathcal{H}(\Sigma_\eps\times\R)&:=\left\{\psi\in H^1_{loc}(\Sigma_\eps\times\R):\,|\nabla_\Sigma\psi|,\,|\partial_t\psi|,\,p_0(\eps\cdotp)^{1/2}\psi\in L^2(\Sigma_\eps\times\R)\right\}\\
\mathcal{X}&:=\left\{\psi\in\mathcal{H}(\Sigma_\eps\times\R):\,\int_\R \psi(\y,t)v'_\star(t)dt=0,\,\forall\,\y\in\Sigma_\eps\right\}\\
p_0^{-1/2}L^2(\Sigma)&:=\{f\in L^2_{loc}(\Sigma):\,p_0^{1/2}f\in L^2(\Sigma)\}
\end{aligned}
\end{equation}
and the linear operator
$$\tilde{L}_\eps:\mathcal{X}\to p_0^{-1/2}L^2(\Sigma)$$
defined by
$$\tilde{L}_\eps\psi(\eps\y):=\int_\R (L+\Delta_{\Sigma_\eps}+a^{ij}\partial_{ij}+\eps b^i\partial_i)\psi(\y,t)v'_\star(t)\,dt,\qquad\forall\,\y\in\Sigma_\eps.$$
The operator $\tilde{L}_\eps\psi$ represents the projection of $L|_{\mathcal{X}}$ onto $v'_\star$. In fact, given $\psi\in \mathcal{X}$, the terms $\Delta_{\Sigma_\eps}\psi$ and $a^{ij}\psi_{ij}+\eps b^i\psi_i$ are orthogonal to $v'_\star$. For this reason the terms involving the second derivatives in $\Sigma_\eps$ do not contribute to the projection onto $v'_\star$.\\

Similarly, we introduce the projection of the operator ${\tt L}_\eps$ onto $v'_\star$, that is the operator $$\tilde{\tt L}_\eps:\mathcal{H}(\R^{N+1})\to p_0^{-1/2}L^2(\Sigma)$$
given by
$$\tilde{\tt L}_\eps\phi(\eps\y):=\int_\R ({\tt L}_\eps\phi)^\sharp(\y,t)v'_\star(t)\,dt,\qquad\forall\,\y\in\Sigma_\eps$$
\begin{lemma}\label{lemma-eq-varphi-R}
In the above notations, $\varphi_{i,k}$ fulfills the equation
\begin{equation}\label{eq-varphi-R}
-J_\Sigma \varphi_{i,k}=  \eps^{-2}\lambda_{i,k}p_0\varphi_{i,k}+\eps^{-2}f_\eps\varphi_{i,k} -c_\star(a^{lj}(\varphi_{i,k})_{lj}+b^j(\varphi_{i,k})_j)+\eps^{-2}\tilde{L}_\eps\psi_{i,k}^\bot+\eps^{-2}\tilde{\tt L}_\eps\phi_{i,k},
\end{equation}
where $\mu_{i,k}=\eps^{-2}\lambda_{i,k}$,
$f_\eps$ fulfills
$$|f_\eps(y)|\le c\eps^3(1+|y|)^{2(\nu-1)},\qquad\forall\,y\in\Sigma,$$
and
$$|\tilde{L}_\eps\psi_{i,k}^\bot(\eps^{-1}y)|\le c\eps^{2}\left((1+|y|)^{-2}\|\psi^\bot_{i,k}(\eps^{-1}y,\cdotp)\|_{L^2_t(\R)}+(1+|y|)^{\nu-1}\|\nabla_{\Sigma_\eps}\psi^\bot_{i,k}(\eps^{-1}y,\cdotp)\|_{L^2_t(\R)}\right).$$
\end{lemma}
\begin{proof}
The idea is the following: we multiply \eqref{eq-parallel-bot} by $v'_\star$ and integrate by parts over $\R$. First note that, by Lemma \ref{lemma:quasi_autofunzione_Allen-Cahn}, we have
\begin{equation}
\int_\R L(\varphi_{i,k}(\eps\y)v'_\star)v'_\star\,dt=-\eps^2c_\star J_\Sigma\varphi_{i,k}(\eps\y)+f_\eps(\eps\y)\varphi_{i,k}(\eps \y)-\eps^2c_\star(a^{lj}(\varphi_{i,k})_{lj}+b^j(\varphi_{i,k})_j)(\eps\y)\qquad\forall\,\y\in\Sigma_\eps,
\label{pr-Sigma-R}
\end{equation}
where $f_\eps\in C^{0,\alpha}(\Sigma)$ is such that
\begin{equation}\notag
|f_\eps(y)|\le c\eps^3(1+|y|)^{2(\nu-1)}\qquad\forall\, y\in\Sigma.  \end{equation}
for some constant $c>0$ independent of $k$.\\

Now let us consider the term involving $\psi^\bot:=\psi_{i,k}^\bot$. Using the expansion of the Laplacian in shifted Fermi coordinates \eqref{Laplacian_shifted_Fermi}, we note that
\begin{equation}\label{exp-psi-bot}
\begin{aligned}
-\Delta_{\y,t} \psi^\bot+W''(v_\eps)\psi^\bot=&-\Delta_{\Sigma_\eps} \psi^\bot-\partial^2_t\psi^\bot+\eps^2 J_\Sigma h \partial_t\psi^\bot+\eps^2|A_\Sigma|^2t\partial_t\psi+2\eps\nabla_\Sigma h\cdotp\nabla_\Sigma\partial_t\psi^\bot\\
&-\eps^2|\nabla_\Sigma h|^2\partial^2_t\psi^\bot+\eps{\tt Q}\partial_t\psi^\bot-(a^{lj}\partial_{lj}\psi^\bot+\eps b^i\partial_j\psi^\bot)\\
&+(a^{lj}h_{lj}+b^j h_j)\partial_t\psi^\bot+W''(v_\star)\psi^\bot+(W''(v_\eps)-W''(v_\star))\psi^\bot.
\end{aligned}
\end{equation}    
As a consequence, multiplying by $v'_\star$ and integrating over $\R$, we can see that the function defined by
$$g(\eps\y):=\int_{\R}(-\Delta_{\y,t} \psi^\bot+W''(v_\eps)\psi^\bot)v'_\star \, dt,\qquad\forall\,\y\in\Sigma_\eps,$$
satisfies 
$$|g_\eps(y)|\le c\eps^2\left((1+|y|)^{-2}\|\psi^\bot(\eps^{-1}y,\cdotp)\|_{L^2_t(\R)}+(1+|y|)^{\nu-1}\|\nabla_{\Sigma_\eps}\psi^\bot(\eps^{-1}y,\cdotp)\|_{L^2_t(\R)}\right).$$
\end{proof}
We are now interested in estimating the $H^1(\Sigma_\eps\times\R)$-norm of $\psi_R^\bot$. More precisely, we will see that such a norm is small in $\eps$ and bounded uniformly in $R$, which means that $\psi_R$ is basically parallel to $v'_\star$.
\begin{lemma}\label{lemma-improve-decay}
There exists $c>0$ such that
\begin{equation}\notag
\|\psi^\bot_{i,k}\|_{H^1(\Sigma_\eps\times\R)}\le c\eps^2,
\end{equation}
for any $1\le i\le n+1$ and $k>0$ large enough.
\end{lemma}

\begin{proof}
The idea is the following: this time we multiply \eqref{eq-parallel-bot} by $\psi_{i,k}^\bot$ and integrate by parts over $\Sigma_\eps\times\R$.\\

Using that
$$\int_{\R}((v')^2+W''(v_\star)v^2)dt\ge c\int_\R v^2 dt\qquad\forall\,v\in H^1(\R),\,\int_\R v v'_\star dt=0,$$
and the expansion \eqref{exp-psi-bot}, it is possible to see that there exists $c>0$ such that
$$\int_{\Sigma_\eps\times \R}(-\Delta_{\y,t} \psi^\bot_{i,k}+W''(v_\eps)\psi^\bot_{i,k})\psi^\bot_{i,k} \,dtd\sigma\ge c\|\psi^\bot_{i,k}\|^2_{H^1(\Sigma_\eps\times\R)},\qquad\forall\,1\le i\le n+1,\,k>0$$
To deal with term involving $\varphi_{i,k}$, we use the expansion \eqref{exp-L-psi-parallel} and integrate by parts to prove that there exists $c>0$ such that
$$\left|\int_{\Sigma_\eps\times\R} L(\varphi_{i,k}(\eps\y)v'_\star)\psi^\bot_{i,k}\,d\sigma dt\right|\le c\eps^{2-N}\|p_0^{1/2}\psi^\bot_{i,k}\|_{L^2(\Sigma_\eps\times\R)}\|\varphi_{i,k}\|_{\mathcal{H}(\Sigma)},$$
for any $k>0$ large enough.\\

Multiplying equation \eqref{eq-varphi-R} by $\varphi_{i,k}$, integrating by parts over $\Sigma$ and using Lemma \ref{lemma-varphi-R-bd} , we can see that $\eps^{-N/2}\varphi_R$ is bounded in $\mathcal{H}(\Sigma)$, uniformly in $\eps$ and $R$. Therefore we have
$$\left|\int_{\Sigma_\eps\times\R} L(\varphi_{i,k}(\eps\y)v'_\star)\psi^\bot_{i,k}\,d\sigma dt\right|\le c\eps^2\|p_0^{1/2}\psi^\bot_{i,k}\|_{L^2(\Sigma_\eps\times\R)}\le c\eps^2\|\psi_{i,k}^\bot\|_{L^2(\Sigma_\eps\times\R)},$$
for some constant $c>0$. Hence the conclusion follows using that $\lambda_{i,k}<0$ and the smallness of ${\tt R}_\eps$.
\end{proof}
\begin{remark}
\item 
Lemma \ref{lemma-improve-decay} is actually very important, since it represents a relevant improvement of decay with respect to Lemma \ref{lemma-varphi-R-bd}. In fact Lemma \ref{lemma-varphi-R-bd} provides a uniform bound of $\psi_{i,k}^\bot$ in $\mathcal{H}(\Sigma_\eps\times\R)$, while Lemma \ref{lemma-improve-decay} provides a uniform bound in a stronger norm. Moreover, the upper bound behaves well in $\eps$, in the sense that it is small enough for our purposes.    
\end{remark}
\begin{lemma}
There exists $c>0$ such that 
\begin{equation}\notag
\|\psi^\bot_{i,k}\|_{L^\infty(\Sigma_\eps\times\R)} \le c,  
\end{equation}
for any $1\le i\le n+1$ and $k>0$ large enough.
\end{lemma}
\begin{proof}
Let us assume by contradiction that, up to a subsequence, $\psi_k:=\psi_{i,k}^\bot$ satisfy $$\mu_k:=\|\psi_k\|_\infty\to\infty,\qquad k\to\infty.$$ 
Let $(\y_k,t_k)\in\Sigma_{\eps_k}\times\R$ be such that $\|\psi_k\|_\infty=|\psi_k(\y_k,t_k)|$. Due to the maximum principle and the form $u_{\eps_k}$, we can see that $t_k$ is bounded. Hence, up to a subsequence, $t_k\to t_\infty\in\R$. Let $\delta>0$ and let $$Y_k:z\in B_\delta(0)\subset\R^N\mapsto y:=Y(z)\in\Sigma$$ 
be a parametrization of $\Sigma$ around $\eps_k\y_k\in\Sigma$ such that $Y_k(0)=\eps_k\y_k$ and the corresponding local metric ${\tt g}$ on $\Sigma$ satisfies ${\tt g}_{ij}(0)=\delta_{ij}$. In these coordinates, the Laplace-Beltrami operator of $\Sigma_{\eps_k}$ around $\y_k$ reads $$\Delta_{\Sigma_{\eps_k}}={\tt g}^{ij}(\eps_k {\tt z})\partial_{ij}+\frac{\partial_j(\sqrt{|{\tt g}(\eps_k {\tt z})|}{\tt g}^{ij}(\eps_k {\tt z}))}{\sqrt{|{\tt g}(\eps_k {\tt z})|}}\partial_i,$$
therefore, thanks to the expansion of the Laplacian in shifted Fermi coordinates \eqref{Laplacian_shifted_Fermi}, the function $\tilde{\psi}_k({\tt z},t):=\frac{\eps_k^{-1}\psi_k(Y(\eps_k {\tt z}),t)}{\mu_k}$ satisfy the problem

\begin{equation}\notag
\begin{aligned}
-{\tt a}^{IJ}(\eps_k {\tt z},t)\partial_{IJ}\tilde{\psi}_k-{\tt b}^I(\eps_k{\tt z},t)\partial_I\tilde{\psi}_k+W''(v'_\star(t))\tilde{\psi}_k&=\tilde{f}_k\qquad\forall\,({\tt z},t)\in B_{\eps_k^{-1}\delta}\times\R, \\  
\int_\R \tilde{\psi}_k({\tt z},t)v'_\star(t)dt&=0\qquad\forall\,{\tt z}\in B_{\eps_k^{-1}\delta}(0),
\end{aligned}    
\end{equation}
where $1\le I,J\le N+1$,
\begin{equation}\notag
\begin{aligned}
&{\tt a}^{ij}(\eps_k {\tt z},t)\to\delta^{ij},\qquad\forall\,1\le i,j\le N\\
&{\tt a}^{tt}(\eps_k {\tt z},t)\to 1\\
&{\tt a}^{it}(\eps_k {\tt z},t)={\tt a}^{ti}(\eps_k {\tt z},t)\to 0,\qquad\forall\, 1\le i\le N\\
&{\tt b}^I(\eps_k {\tt z},t)\to 0,\qquad\forall\, 1\le I\le N+1
\end{aligned}    
\end{equation}
as $k\to \infty$, uniformly in ${\tt z}$ and $t$, and $\tilde{f}_k\to 0$ in $L^2(\R^{N+1})$. Since $\|\tilde{\psi}_k\|_\infty=1$ for any $k$, $\tilde{\psi}_k$ admits a subsequence which converges locally uniformly to a bounded solution $\psi_\infty\in H^1(\R^{N+1})$ to 
\begin{equation}\notag
\begin{aligned}
-\Delta_{\tt z}\psi_\infty-\partial_t^2\psi_\infty+W''(v_\star(t))\psi_\infty&=0\qquad\forall\,({\tt z},t)\in\R^N\times\R \\
\int_\R \psi_\infty({\tt z},t)v'_\star(t)dt&=0\qquad\forall\,{\tt z}\in\R^N.
\end{aligned}    
\end{equation}
Therefore $\psi_\infty\equiv 0$. However, we have
$$1=|\tilde{\psi}_k(0,t_k)|\to|\psi_\infty(0,t_\infty)|,$$
which is a contradiction.\\   
\end{proof}
\color{black}\begin{remark}\label{rem-bound-psi-bot}
\begin{enumerate}
\item The uniform $L^\infty$-bound gives an exponential decay of $\psi_{i,k}^\bot$ which is also uniform in $k$, namely there exist $c,\,\gamma>0$ such that
\begin{equation}\notag
|\psi_{i,k}^\bot(\y,t)|\le ce^{-\gamma/\eps},\qquad \forall\,(\y,t)\in \Sigma_\eps\times\R,
\end{equation}
for any $1\le i\le n+1$ and $k>0$ large enough.
\item As a consequence, we have
\begin{equation}\notag
\begin{aligned}
\|(1-\chi^\natural_1)\phi_{i,k}\|_{H^1(\R^{N+1})}&\le\|\psi_{i,k}^\bot\|_{H^1({\rm supp}(\nabla_{\y,t}\chi_1))}+ce^{-\gamma/\eps}\eps^{-N/2}\|p_0^{1/2}\varphi_{i,k}\|_{L^2(\Sigma)}\\
&\le ce^{-\gamma/\eps}+ce^{-\gamma/\eps}\eps^{-N/2}\|p_0^{1/2}\varphi_{i,k}\|_{L^2(\Sigma)}\le ce^{-\gamma/\eps},
\end{aligned}
\end{equation}
for any $1\le i\le n+1$ and $k>0$ large enough.
\item Testing equation \eqref{eq-varphi-R} with $\varphi_{i,k}$ and using the estimates about $\psi_{i,k}^\bot$ and $\phi_{i,k}$ given by Lemma \ref{lemma-improve-decay} and \eqref{est-phi-far} respectively, we can see that there exists $c>0$ such that 
\begin{equation}\label{est-varphi-R-H}
\|\varphi_{i,k}\|_{\mathcal{H}(\Sigma)}\le c\|p_0^{1/2}\varphi_{i,k}\|_{L^2(\Sigma)}\le c,
\end{equation}
for any $k>0$ large enough.
\end{enumerate}
\end{remark}
\begin{lemma}\label{lemma-lower-est-ev}
In the above notations, there exists $\bar{\mu}>0$ such that
\begin{equation}\notag
-\bar{\mu}\eps^{2}\le \lambda_{i,k}<0,\qquad\forall\,1\le i\le n+1.
\end{equation}   
for any $k>0$ large enough.
\end{lemma}
\begin{proof}
Testing equation \eqref{tr-prob-BR} with $\phi_{i,k}$ we can see that
\begin{equation}\notag
\lambda_{i,k}\int_{\R^{N+1}}p(x)\phi_{i,k}^2 \,dx=\Q_{u_\eps}(\phi_{i,k}).
\end{equation}
Given a measurable subset $B\subset\R^{N+1}$, we set
$$\Q_{u_\eps,B}(\phi):=\int_{B}(|\nabla \phi|^2+W''(u_\eps)\phi^2)dx,\qquad\forall\,\phi\in H^1(\R^{N+1})$$
Splitting the integrals and using that $W''(u)\ge\frac{W''(a)}{2}>0$ in $\R^{N+1}\setminus{\rm supp}(\chi_1^\natural)$, we can see that
$$\Q_{u_\eps}(\phi_{i,k})=\Q_{u_\eps,{\rm supp(\chi_1^\natural)}}(\phi_{i,k})+\Q_{u_\eps,\R^{N+1}\setminus{\rm supp(\chi_1^\natural)}}(\phi_{i,k})\ge \Q_{u_\eps,{\rm supp(\chi_1^\natural)}}(\phi_{i,k}).$$
Using the change of variables induced by the shifted Fermi coordinates, we can see that $\phi_{i,k}(x)=\psi_{i,k}(\y,t)$ on the support of $\chi_1^\natural$. \\

Using the estimate for $\psi_{i,k}^\bot$ and the exponential decay of $v'_\star$ in $t$, we have
\begin{equation}\notag
\left|\Q_{u_\eps,{\rm supp(\chi_1^\natural)}}(\phi_{i,k})-\int_{\Sigma_\eps\times\R}L\psi_{i,k}\psi_{i,k}\,d\sigma dt\right|\le c\|\psi^\bot_{i,k}\|_{H^1(\Sigma_\eps\times\R)}^2\le c\eps^4
\end{equation}
Using the decomposition $\psi_{i,k}=\psi_{i,k}^\parallel+\psi_{i,k}^\bot$, we can see that
$$\int_{\Sigma_\eps\times\R}L\psi_{i,k}\psi_{i,k}\,d\sigma dt=\int_{\Sigma_\eps\times\R}L\psi_{i,k}^\parallel\psi_{i,k}^\parallel\,d\sigma dt+2\int_{\Sigma_\eps\times\R}L\psi_{i,k}^\parallel\psi_{i,k}^\bot\,d\sigma dt+\int_{\Sigma_\eps\times\R}L\psi_{i,k}^\bot\psi_{i,k}^\bot\,d\sigma dt.$$
Due to the orthogonality condition satisfied by $\psi_{i,k}^\bot$ and the nondegeneracy of $v_\star$ (see Section $4$) we can see that
$$\int_{\R}((\partial_t\psi_{i,k}^\bot)^2+W''(v_\star)(\psi^\bot_{i,k})^2)dt\ge c\int_\R(\psi^\bot_{i,k})^2 dt\ge 0,$$
for some constant $c>0$.\\

Due to the expansion of the Laplacian in shifted Fermi coordinates \eqref{Laplacian_shifted_Fermi}, arguing as in the proof of Lemma \ref{lemma-improve-decay} we can control the mixed term as follows
$$\left|\int_{\Sigma_\eps\times\R}L\psi_{i,k}^\parallel\psi_{i,k}^\bot\,d\sigma dt\right|\le c\eps^{2-N}\|\psi_{i,k}^\bot\|_{L^2(\Sigma_\eps\times\R)}\|p_0^{1/2}\varphi_{i,k}\|_{L^2(\Sigma)}\le c\eps^2\|p^{1/2}\phi_{i,k}\|_{L^2(\R^{N+1})}.$$
Note that integration by parts in $\Sigma$ is justified by the fact that $\phi_{i,k}\in H^1_0(B_{R_k}(0))$, which yields that $\varphi_{i,k}$ and $\psi_{i,k}^\bot$ are compactly supported.\\

Similarly, the term involving $\psi_{i,k}^\parallel$ fulfills
$$\int_{\Sigma_\eps\times\R}L\psi_{i,k}^\parallel\psi_{i,k}^\parallel\,d\sigma dt=-\eps^{2-N} c_\star\int_\Sigma J_\Sigma\varphi_{i,k}\varphi_{i,k} d\sigma+R_\eps,$$
where
$$|R_\eps|\le c\eps^{3-N}\|\varphi_{i,k}\|_{\mathcal{H}(\Sigma)}\le c\eps^{3-N}\|p_0\varphi_{i,k}\|_{L^2(\Sigma)}\le c\eps^3\|p^{1/2}\phi_{i,k}\|_{L^2(\R^{N+1})}^2$$
and $c_\star:=\int_{\R}(v'_\star)^2dt>0$.\\

Using the upper estimate for $\|\varphi_{i,k}\|_{\mathcal{H}(\Sigma)}$ provided in \eqref{est-varphi-R-H} and the decay of $|A_\Sigma|^2$, namely the fact that $$|A_\Sigma|^2\le c(1+|y|)^{-2}\qquad\forall\, y\in\Sigma,$$
we can see that there exists $c>0$ such that $$\left|\int_\Sigma -J_\Sigma\varphi_{i,k}\varphi_{i,k} d\sigma\right|\le c\|\varphi_{i,k}\|^2_{\mathcal{H}(\Sigma)} \le c\|p_0\varphi_{i,k}\|^2_{L^2(\Sigma)}.$$
As a consequence, recollecting all the above information,
\begin{equation}\notag
\begin{aligned}
\Q_{u_\eps}(\phi_{i,k})&\ge
c_\star\eps^{2-N}\int_{\Sigma}-J_\Sigma \varphi_{i,k}\varphi_{i,k} d\sigma dt-c\eps^2\int_{\R^{N+1}}p(x)\phi^2_{i,k}dx\\
&\ge-\eps^{2-N}c_\star\int_\Sigma p_0\varphi_{i,k}^2 d\sigma -c\eps^2\int_{\R^{N+1}}p(x)\phi^2_{i,k}dx\ge-c\eps^2\int_{\R^{N+1}}p(x)\phi^2_{i,k}dx,
\end{aligned}
\end{equation}
for some constant $c>0$.
\end{proof}

\begin{remark}\label{rem-bound-varphi-R}
Since $\|p_0^{1/2} \psi^\bot_{i,k}\|_{L^2(\Sigma_\eps\times\R)}\le c\eps^2$ and $\|p^{1/2}\phi_{i,k}\|_{L^2(\R^{N+1})}=1$, due to Lemma \ref{lemma-varphi-R-bd}, there exists $\sigma>0$ such that
\begin{equation}\label{low-est-varphi}
0<\sigma\le\eps^{-N/2}\|\varphi_{i,k}p_0^{1/2}\|_{L^2(\Sigma)}\le \sigma^{-1},
\end{equation}
for any $1\le i\le n+1$ and $k>0$ large enough. So in particular $\varphi_{i,k}\ne 0$.
\end{remark}
Now we use all the previous information to pass to the limit as $k\to\infty$ and construct a Jacobi field $\varphi\ne 0$ of $\Sigma$ which is is $D_\nu(\Sigma)$ for some $\nu>\bar{\nu}$.
\begin{lemma}\label{lemma-convergence}
There exist $\mu_1,\dots,\mu_{n+1}\in[-\mu,0]$ and $n+1$ solutions $\varphi_{1},\dots,\varphi_{n+1}\in\mathcal{H}(\Sigma)$ to
$$-J_\Sigma\varphi_{i}=\mu_{i}p_0(y)\varphi_{i}\qquad\text{in}\,\Sigma$$
such that $\varphi_{i,k}\to\varphi_i$ weakly in $\mathcal{H}(\Sigma)$ and $\eps^{-2}\lambda_{i,k}\to\mu_i$ as $k\to\infty$.
\end{lemma}
\begin{proof}
From Lemma \ref{lemma-lower-est-ev}, we have $\eps^{-2}\lambda_{i,k}\in[-\bar{\mu},0]$, hence there exists a subsequence such that $\eps^{-2}\lambda_{i,k}\to\mu_i$, for some $\mu_i\in[-\bar{\mu},0]$. Moreover, due to \eqref{est-varphi-R-H}, $\varphi_{i,k}$ is bounded in $\mathcal{H}(\Sigma)$, therefore, up to a subsequence, $\varphi_{i,k}$ converges weakly to some $\varphi_i\in\mathcal{H}(\Sigma)$.\\

It remains to deal with the smaller terms, due to the smallness of $\psi_{i,k}^\bot$ given by Lemma \ref{lemma-improve-decay} and the estimate of $\tilde{L}_\eps\psi^\bot_{i,k}$ given by Lemma \ref{lemma-eq-varphi-R}, we can see that
$$\|\eps^{-2}p_0^{-1/2}\tilde{L}_\eps\psi_{i,k}^\bot\|_{L^2(\Sigma)}\le c\|\psi_{i,k}^\bot\|_{H^1(\Sigma_\eps\times\R)}\le c\eps^2\to 0,\qquad k\to\infty.$$
Similarly, using the smallness of $\phi_{i,k}$ far from $\Sigma_\eps$, given by \eqref{est-phi-far}, we can see that
$$\|\eps^{-2}p_0^{-1/2}\tilde{\tt L}_\eps\phi_{i,k}\|\le ce^{-\tilde{\gamma}/\eps}\to 0,\qquad k\to\infty,$$
for some $\tilde{\gamma}\in(0,\gamma)$. Moreover, using the decay of $f_\eps$ given by \eqref{est-f-eps}, we have
$$\|\eps^{-2}p_0^{-1/2}f_\eps\varphi_{i,k}\|_{L^2(\Sigma)}\le c\eps^{1-N}\|p_0^{1/2}\varphi_{i,k}\|_{L^2(\Sigma)}\le c\eps\to 0,\qquad k\to\infty.$$
Using the decay of the coefficients $a^{lj}$ and $b^j$ in $\eps$ and in $\y$, we finally conclude that $\varphi_i$ is a weak solution to the equation $$-J_\Sigma\varphi_i=\mu_i p_0(y)\varphi_i\qquad\text{in}\,\Sigma.$$
\end{proof}
Now we exhibit a careful analysis to show that the weak limits $\varphi_{i}$ of $\varphi_{i,k}$ are nontrivial, for any $1\le i\le n+1$. For $\mu\in\R$, we introduce the operator $$J_\Sigma^\mu:=J_\Sigma+\mu p_0$$
and we define its indicial roots as
$$-\frac{N-2}{2}\pm{\rm Re}(\Lambda_j(\mu)),$$
where we have set
$$\Lambda_j(\mu):=
\begin{cases}
\sqrt{\left(\frac{N-2}{2}\right)^2-\mu+\lambda_j}\qquad\text{if}\,\left(\frac{N-2}{2}\right)^2-\mu+\lambda_j\ge 0\\
i\sqrt{-\left(\left(\frac{N-2}{2}\right)^2-\mu+\lambda_j\right)}\qquad\text{otherwise}.
\end{cases}$$
Note that $\Lambda_j(0)=\Lambda_j$ and $\Lambda_j(\mu)^2\in\R$ may be negative, since $\Lambda_j(\mu)$ are not necessarily real numbers.\\

In these notations, taking $\delta\in(2\nu, 0]$ and $\mu:=\eps^{-2}\lambda_{i,k}$, we have $p_0^{-1/2} J_\Sigma^\mu\varphi_{R}\Gamma_{-\delta}\in L^2(\Sigma)$
and 
\begin{equation}\label{improve-decay-delta}
\|p_0^{-1/2}J_\Sigma^\mu\varphi_{i,k}\Gamma_{-\delta}\|_{L^2(\Sigma)}\le c\|\psi_{i,k}^\bot\|^2_{H^1(\Sigma_\eps\times\R)}+c\eps^{1-N}\|\varphi_{i,k}\|_{\mathcal{H}(\Sigma)}+ce^{-\tilde{\gamma}/\eps}
\to 0\qquad\text{as}\,\eps\to 0.
\end{equation} 
Roughly speaking, the decay estimates about the remainders done in the proof of Lemma \ref{lemma-convergence} can be improved, in the sense that we can gain some arbitrarily small extra decay, thanks to the fact that $\nu<0$.

\begin{lemma}\label{rem-compact}
Let $\mu\in\R$ and $J_\Sigma^\mu:=J_\Sigma+\mu p_0(y)$. Let $\delta\in\R\setminus\{\pm{\rm Re}(\Lambda_j(\mu))\}_{j\ge 0}$. Then there exist $R_0,\,c>0$ such that   $$\|p_0^{1/2}\varphi\Gamma_{-\delta}\|_{L^2(\Sigma)}\le c(\|p_0^{-1/2}J_\Sigma^\mu\varphi\Gamma_{-\delta}\|_{L^2(\Sigma)}+\|\varphi\|_{L^2(K_{R_0})}),\qquad\forall\varphi \in H^1_0(K_R),\,\forall\, R>R_0.$$ 
\end{lemma}
\begin{proof}
This can be proved by arguing in the same way as in Remark $2.1$ of \cite{agudelo2025fredholm}. The only difference are the fact that $\varphi$ is defined in a bounded set and fulfills Dirichlet boundary conditions, which make the integrations by parts possible, and the presence of $\mu$. However, for the sake of completeness, we give a sketch of proof.\\

First we recall that the equation 
\begin{equation}\label{Dirichlet-K_R}
\begin{aligned}
-J_\Sigma^\mu\varphi&=f\qquad\text{in}\,K_R\\
u&=0\qquad\text{on}\,\partial K_R
\end{aligned}
\end{equation}
is equivalent to 
\begin{equation}\notag
\begin{aligned}
-\mathcal{L}^\mu u=g\qquad\text{in}\,\Sigma\\
u=0\qquad\text{on}\,\partial K_R,
\end{aligned}    
\end{equation}
where the operator
$$\mathcal{L}u:=|y|^N{\rm div}(|y|^{2-N}\nabla_\Sigma u)+\left(|y|^2|A_\Sigma|^2+|y|^{\frac{N+2}{2}}\Delta_\Sigma\left(|y|^{-\frac{N-2}{2}}\right)\right)u$$
is defined in \cite{agudelo2024jacobi} and
$$u:=|y|^{\frac{N-2}{2}}\phi,\quad g:=|y|^{\frac{N+2}{2}}f,\quad\mathcal{L}^\mu:=\mathcal{L}+\mu|y|^2|A_\Sigma|^2.$$
If $R_0>1$ is large enough, then $\Sigma_{R_0-1}=\Sigma\setminus K_{R_0-1}$ can be expressed as a normal graph over $C$, namely any point $y\in \Sigma_{R_0-1}$ can be written as
$$y=e^t\theta+w(e^t,\theta)\nu_C(\theta),\qquad\forall\,(\theta,t)\in\Gamma\times(\log(R_0-1),\infty).$$
We recall that, using the change of variables $u(y)=\tilde{u}(\theta,t)$ in $\Sigma_{K_{R_0-1}}$, the operator $\mathcal{L}$ can be expanded as
$$\mathcal{L}u=\partial_t^2\tilde{u}+\Delta_\Gamma\tilde{u}+|A_\Gamma|^2\tilde{u}-\left(\frac{N-2}{2}\right)^2\tilde{u}+\mathcal{R}\tilde{u}:=L\tilde{u}+\mathcal{R}\tilde{u},$$
where the $\mathcal{R}$ is small in the sense of formula $(3.3)$ of \cite{agudelo2024jacobi}.\\

Let $\eta:\R\to[0,1]$ be a smooth cutoff function such that $\eta=1$ in $(\log R_0,\infty)$ and $\eta=0$ in $(-\infty,\log(R_0-1))$. Therefore, if $\varphi$ satisfies \eqref{Dirichlet-K_R}, $u:=|y|^{\frac{N-2}{2}}\phi$ and $g:=|y|^{\frac{N+2}{2}}f$, the corresponding function $\bar{u}(\theta,t):=\eta(t)\tilde{u}(\theta,t)$ satisfies the homogeneous Dirichlet problem
\begin{equation}\notag
\begin{aligned}
-L\bar{u}-\mu\bar{u}&=\bar{g}  \qquad\text{in}\,\mathring{K}_{R}\setminus K_{R_0-1}\\
\bar{u}&=0\qquad\text{on}\,\partial (\mathring{K}_{R}\setminus K_{R_0-1}),
\end{aligned}
\end{equation}
where $\bar{g}:=\eta \tilde{g}+2\partial_t\eta\partial_t\tilde{u}+\partial^2_t\eta \tilde{u}+\mathcal{R}\tilde{u}$ and $\tilde{g}(\theta,t):=g(y)$.\\

Taking an orthonormal basis $\{\varphi_j\}_{j\ge 0}$ of $L^2(\Gamma)$ of eigenvectors of $-J_\Gamma$ corresponding to the eigenvalues $$\lambda_0<\lambda_1\le \dots\le \lambda_j\to\infty,\qquad j\to\infty,$$
we can expand $\bar{u}$ in Fourier series in the form
$$\bar{u}(\theta,t)=\sum_{j=0}^\infty u_j(t)\varphi_j(\theta), \qquad  u_j(t):=\int_\Gamma \bar{u}(\theta,t)\varphi_j(\theta)d\theta.$$
It is possible to see that the Fourier components satisfy the problems
\begin{equation}\label{eq-Fourier}
\begin{aligned}
-L_j u_j&=g_j \qquad\text{in}\,(t_0,t_R)\\
u_j(t_0)=u_j(t_R)&=0,
\end{aligned}
\end{equation}
where $$L_j:=\partial^2_t+\mu-\left(\frac{N-2}{2}\right)^2-\lambda_j=\partial^2_t-\Lambda_j(\mu)^2,\qquad t_0=\log(R_0-1),\,t_R=\log R$$
and $$g_j(t):=\int_\Gamma \bar{g}(\theta,t)\varphi_j(\theta)d\theta$$
are the Fourier components of the right-hand side.\\

Taking $\delta\in\R$ and testing equation \eqref{eq-Fourier} with $u_j e^{2\delta t}$ we get the upper bound
$$\left(\Lambda_j(\mu)^2-\delta^2\right)^2\int_\R u_j^2 e^{2\delta t}dt\le \int_\R g_j^2 e^{2\delta t}dt.$$
For the computations we refer to the proof of Proposition $6.2.1$ of \cite{pacard2008connected}. The only differences are that our functions $u_j$ are defined on a bounded set and they vanish on the boundary.\\

As a consequence, if $\delta\notin\{\pm{\rm Re}(\Lambda_j(\mu))\}$, taking the sum of the Fourier series we get the upper bound
$$\int_{K_R\setminus\Sigma_{R_0-1}}u^2 \Gamma_{2\delta}d\sigma\le c\int_\Sigma g^2\Gamma_{2\delta}d\sigma+c_\delta\int_{K_{R_0}}u^2d\sigma,$$
which concludes the proof.
\end{proof}
\begin{lemma}\label{lemma-phi-eps-ine0}
In the above notations, there exist $k_0>0$ such that, for any $k>k_0$, $\|\varphi_{i,k}\|_{L^2(K_{R_0})}\ge\sigma/2>0$.
\end{lemma}
\begin{proof}
Due to Lemma \ref{lemma-varphi-R-bd}, there exists $\sigma>0$ such that $\|p_0^{1/2}\varphi_{i,k}\|_{L^2(\Sigma)}\ge \sigma$. As we observed in \eqref{improve-decay-delta}, if we take $\delta\le 0$ wit $|\delta|$ small enough (more precisely $\delta\in(2\nu,0]$), we have
$$\|p_0^{-1/2}J^{\mu_i}_\Sigma\varphi_{i,k}\Gamma_{-\delta}\|_{L^2(\Sigma)}\to 0,\qquad k\to \infty,$$
where we have set $\mu_{i,k}:=\eps^{-2}\lambda_{i,k}$.\\

Therefore, due to Lemma \ref{rem-compact}, if  $\delta\in(2\nu,0]\setminus\{\pm{\rm Re}(\Lambda_j(\mu_i))\}_{j\ge 0}$, where $\mu_i$ is defined in Lemma \ref{lemma-convergence}, we have
\begin{equation}\notag
\begin{aligned}
&0<\sigma\le \|p_0^{1/2}\varphi_{i,k}\|_{L^2(\Sigma)}\le \|p_0^{1/2}\varphi_{i,k}\Gamma_{-\delta}\|_{L^2(\Sigma)}\le c(\|p_0^{-1/2}J_\Sigma^{\mu_i}\varphi_{i,k}\Gamma_{-\delta}\|_{L^2(\Sigma)}+\|\varphi_{i,k}\|_{L^2(\Sigma_{R_0})})\\
&\le c(\|p_0^{-1/2}J_\Sigma^{\mu_{i,k}}\varphi_{i,k}\Gamma_{-\delta}\|_{L^2(\Sigma)}+|\mu_{i,k}-\mu_i|\|p_0^{1/2}\varphi_{i,k}\Gamma_{-\delta}\|_{L^2(\Sigma)}+\|\varphi_{i,k}\|_{L^2(\Sigma_{R_0})})\\
&=c(\|p_0^{-1/2}J_\Sigma^{\mu_{i,k}}\varphi_{i,k}\Gamma_{-\delta}\|_{L^2(\Sigma)}+\|\varphi_{i,k}\|_{L^2(\Sigma_{R_0})}+o_k(1)\|p_0^{1/2}\varphi_{i,k}\Gamma_{-\delta}\|_{L^2(\Sigma)}).
\end{aligned}   
\end{equation}
As a consequence, we can see that $\|\varphi_{i,k}\|_{L^2(K_{R_0})}\ge \sigma/2>0$ for any $1\le i\le n+1$ and $k$ large enough. 
\end{proof} 
\begin{lemma}\label{lemma-varphi-eps-ne0}
In the above notations, $\varphi_i\ne 0$, for any $1\le i\le n+1$.
\end{lemma}
\begin{proof}
It follows from the fact that $\varphi_{i,k}\to\varphi_{i}$ in strongly $L^2_{loc}(\Sigma)$ and Lemma \ref{lemma-phi-eps-ine0} holds, so that $\|\varphi_i\|_{L^2(K_{R_0})}\ge \sigma/2$.
\end{proof}

Using the orthogonality condition \eqref{ort-phi-R-ij} and passing to the limit as $k\to\infty$, it is possible to see that
\begin{equation}\label{ort-negative-directions}
\int_\Sigma p_0(y)\varphi_i\varphi_j d\sigma=0,\qquad\forall\,1\le i,j\le n+1,\,i\ne j,    
\end{equation}
so in particular $\{\varphi_i\}_{1\le i\le n+1}$ are linearly independent. Without loss of generality, we can assume that $\mu_i\le \mu_{i+1}\le 0$, for any $1\le i\le n$.
\begin{lemma}\label{lemma-mu=0}
In the previous notations, we have $\mu_{n+1}=0$.    
\end{lemma}
\begin{proof}
Assume by contradiction that $$\mu_{n+1}=\mathcal{R}(\varphi_{n+1}):=\frac{\mathcal{Q}_\Sigma(\varphi_{n+1})}{\|p_0^{1/2}\varphi_{n+1}\|_{L^2(\Sigma)}}<0.$$ 
Let $\chi:\R\to[0,1]$ be a smooth function such that $\chi=1$ in $(-\infty,1)$ and $\chi=0$ in $(2,\infty)$ and let $\chi_R(y):=\chi(|y|/R)$. Arguing as in Lemma $3.4$ of \cite{agudelo2024jacobi}, we can see that $\varphi_i\in C^{2}(\Sigma)$ are classical solutions to $J_\Sigma^{\mu_i}\varphi_i=0$. Then $\tilde{\varphi}_{i,R}=\chi_R\varphi_i$ is in $C^2_c(\Sigma)$ and converges to $\varphi_i$ strongly in $\mathcal{H}(\Sigma)$ as $R\to\infty$. As a consequence, for any $\sigma>0$, there exists $R>0$ such that $$\mathcal{R}(\tilde{\varphi}_{i,R})<\mu_i+\sigma/3\le \mu_{n+1}+\sigma/3<0\qquad\forall\, 1\le i\le n+1$$
if $\sigma$ is so small that $\mu_{n+1}+\sigma<0$.\\

Using the density of $C^\infty_c(\Sigma)$ in $C^2_c(\Sigma)$, we can find $\hat{\varphi}_1,\dots,\hat{\varphi}_{n+1}\in C^\infty_c(\Sigma)$ such that
$$\|\tilde{\varphi}_{i,R}-\hat{\varphi}_i\|_{\mathcal{H}(\Sigma)}<\sigma/3,\qquad\mathcal{R}(\hat{\varphi}_i)<\mu_i+2\sigma/3\le \mu_{n+1}+2\sigma/3<0
,\qquad\forall\,1\le i\le n+1.$$
Finally, using the orthogonality condition \eqref{ort-negative-directions} and the fact that $\|\varphi_{i}-\hat{\varphi}_i\|_{\mathcal{H}(\Sigma)}<\sigma/3$, with the Gram-Schmidt algorithm we can construct $n+1$ functions $\bar{\varphi}_1,\dots,\bar{\varphi}_{n+1}\in C^\infty_c(\Sigma)$ fulfilling the orthogonality condition \eqref{ort-negative-directions} such that $$\mathcal{R}(\bar{\varphi}_i)<\mu_i+\sigma\le \mu_{n+1}+\sigma<0,\qquad\forall\, 1\le i\le n+1.$$
Hence we have found a subspace of $C^\infty_c(\Sigma)$ of dimension $n+1$ of negative directions of $\mathcal{Q}_\Sigma$. This contradicts the fact that $Morse(\Sigma)=n$.
\end{proof}
By Lemma \ref{lemma-mu=0}, $\varphi_{n+1}$ is a Jacobi field of $J_\Sigma$ in $\mathcal{H}(\Sigma)$. Therefore, using Lemma $3.4$ of \cite{agudelo2024jacobi}, we can see that there exists $c>0$ such that
$$|\varphi_{n+1}(y)|\le c(1+|y|)^{-\frac{N-2}{2}},\qquad\forall\,y\in\Sigma,$$
for some constant $c>0$.\\ 

If $\bar{\nu}>-\frac{N-2}{2}$, we have $\varphi_{n+1}\in D_\nu(\Sigma)$ for any $\nu\in(-\frac{N-2}{2},\bar{\nu})$. As a consequence, using that $\Sigma$ is dilation-nondegenerate, we have $D_\nu(\Sigma)=0$, thus $\varphi_{n+1}=0$, which is a contradiction.\\

Note that here the $S$-dilation nondegeneracy is not enough, we actually need dilation-nondegeneracy since $\varphi_{n+1}$ is not necessarily $S$-invariant, even if $u_\eps$ is.\\

If $\bar{\nu}=-\frac{N-2}{2}$, using Lemma $11.1.4$ of \cite{pacard2008connected}, we can see that the exact asymptotic behaviour of $\varphi_{n+1}$ is given by the indicial roots of $J_\Sigma$. In other words, writing
$$y=r\theta+w(r,\theta)\nu_C(\theta),\qquad\forall\,(r,\theta)\in(R_0,\infty)\times\Gamma,$$
we have
$$\varphi_{n+1}(y)=r^{\mu}(a\log r+b)v_j(\theta)(1+o_r(1)),\qquad \text{as}\,r\to\infty,$$
uniformly in $\theta\in\Gamma$, for some $a,\,b\in\R,\,b\ne 0$, and some $\mu\in\{-\frac{N-2}{2}\pm{\rm Re}(\Lambda_j)\}$. Since $\varphi_{n+1}\in\mathcal{H}(\Sigma)$ and $\bar{\nu}=\frac{N-2}{2}$, we necessarily have $$\mu=-\frac{N-2}{2}-{\rm Re}(\Lambda_j)<-\frac{N-2}{2}=\bar{\nu},$$
for some $j\ge 0$. As a consequence, taking $\nu\in(\mu,\bar{\nu})$, we have $\varphi_{n+1}\in D_\nu(\Sigma)$. Since $\Sigma$ is dilation-nondegenerate, we conclude once again that $\varphi_{n+1}=0$, which is a contradiction.

\bibliographystyle{abbrv}
	\bibliography{bibliography}
	\addcontentsline{toc}{chapter}{\protect\numberline{}Bibliography}

\end{document}